\documentclass[10pt]{article}
\usepackage{amsmath}
\usepackage{graphicx}
\usepackage{psfrag}
\usepackage{epsf}
\usepackage{enumitem}

\usepackage{epsf}
\usepackage{xcolor}
\usepackage[normalem]{ulem}

\usepackage[round,longnamesfirst]{natbib}
\usepackage{etoolbox}
\usepackage{mleftright}

\makeatletter
\newcommand\bibstyle@comma{\bibpunct(),a,,}
\newcommand\bibstyle@semicolon{\bibpunct();a,,}
\makeatother

\pretocmd\citet{\citestyle{comma}}\relax\relax
\pretocmd\citep{\citestyle{semicolon}}\relax\relax

\usepackage{url} 
\usepackage{amsthm}
\usepackage{xpatch}
\usepackage{subcaption}
\usepackage{bm}

\usepackage{longtable}
\usepackage{mathrsfs}
\usepackage{tikz}
\usepackage{color} 
\usepackage{amssymb}
\usepackage{latexsym}
\usepackage{xr}
\usepackage{float}
\usepackage{multirow}
\usepackage{dsfont}
\usepackage{booktabs}

\usepackage[ruled]{algorithm2e}
\usepackage{lscape}
\usepackage{rotating}
\numberwithin{equation}{section}

\usepackage{breakcites}
\usepackage{xcolor}
\definecolor{lightblue}{HTML}{044E9E}
\usepackage[breaklinks=true, hyperfootnotes = true,colorlinks,citecolor=lightblue,urlcolor=lightblue, linkcolor=lightblue]{hyperref}

\usepackage{mathtools}
\mathtoolsset{showonlyrefs}

\def\Z{{\mathbb Z}}

\def\R{{\mathbb R}}

\def\cal#1{\mathcal{#1}}
\def\cl#1{\mathcal{#1}}

\def\Leb{{\rm Leb}}

\newcommand{\wt}{\widetilde}
\newcommand{\comment}[1]{}

\newcommand{\ind}{\mathbf{1}}

\def\pp#1{\left(#1\right)}
\newcommand{\proba}{\mathbb P}

\newcommand{\Cov}[0]{\operatorname{Cov}}
\newcommand{\Var}{\operatorname{Var}} 
 
\renewcommand{\P}{\mathbb P}
\newcommand{\esp}{{\mathbb E}}

\newcommand{\var}{{\rm{Var}}}

\newcommand{\MF}{\mathrm{M4}}

\def\pp#1{ \left(#1\right) }
\def\pb#1{ \left[#1\right] }
\def\pc#1{ \left\{#1\right\} }

\usepackage{mleftright}
\usepackage{xparse}

\NewDocumentCommand{\evaluat}{sO{\big}mm}{%
  \IfBooleanTF{#1}
   {\mleft. #3 \mright|_{#4}}
   {#3#2|_{#4}}%
}

\usepackage{hhline} 
\usepackage{highlight}
\renewcommand{\opacity}{50}

\theoremstyle{plain}
\newtheorem{Thm}{Theorem}[section]
\newtheorem{Lem}[Thm]{Lemma}
\newtheorem{Pro}[Thm]{Proposition}
\newtheorem{Cor}[Thm]{Corollary}

\theoremstyle{remark}
\newtheorem{Rem}[Thm]{Remark}
\newtheorem{Def}[Thm]{Definition}

\theoremstyle{definition}
\newtheorem{example}{Example}[section]
\newtheorem{definition}{Definition}[section]

\makeatletter
\xpatchcmd{\proof}{\@addpunct{.}}{\@addpunct{:}}{}{}
\makeatother

\DeclareFontFamily{U}{mathx}{\hyphenchar\font45}
\DeclareFontShape{U}{mathx}{m}{n}{<-> mathx10}{}
\DeclareSymbolFont{mathx}{U}{mathx}{m}{n}
\DeclareMathAccent{\widebar}{0}{mathx}{"73}

\newcommand{\mockalph}[1]{}

\allowdisplaybreaks

\begin{document}

\def\spacingset#1{\renewcommand{\baselinestretch}%
{#1}\small\normalsize} \spacingset{1}

\newtheorem*{assumptionBIC*}{\assumptionnumber}
\providecommand{\assumptionnumber}{}
\makeatletter
\newenvironment{assumptionBIC}[2]
 {%
  \renewcommand{\assumptionnumber}{Assumption #1#2}%
  \begin{assumptionBIC*}%
  \protected@edef\@currentlabel{#1#2}%
 }
 {%
  \end{assumptionBIC*}
 }
\makeatother


\title{On extremes for Gaussian subordination
}

\author{
Shuyang Bai \\ University of Georgia
                            \and
Marie-Christine D\"uker \\ Technical University of Munich}
\date{\today}

\maketitle

\bigskip

\begin{abstract}
\noindent
This paper investigates extreme value theory for processes obtained by applying transformations to stationary Gaussian processes, also called subordinated Gaussian {processes}. The main contributions are as follows: First, we refine the method of \cite{sly2008nonstandard} to allow the covariance of the underlying Gaussian process to decay more slowly than any polynomial rate, nearly matching Berman’s condition. Second, we extend the theory to a multivariate setting, where both the subordinated process and the underlying Gaussian process may be vector-valued, and the transformation is finite-dimensional. In particular, we establish the weak convergence of a point process constructed from the subordinated Gaussian process, from which a multivariate extreme value limit theorem follows.
A key observation that facilitates our analysis, and may be of independent interest, is the following: any bivariate random vector derived from the transformations of two jointly Gaussian vectors with a non-unity canonical correlation always remains extremally independent. This observation also motivates us to introduce and discuss a notion we call $m$-extremal-dependence, which extends the classical concept of $m$-dependence.
Moreover, we relax the restriction to finite-dimensional transforms, extending the results to infinite-dimensional settings via an approximation argument. As an illustration, we establish a limit theorem for a multivariate moving maxima process driven by regularly varying innovations that arise from subordinated Gaussian processes with potentially long memory.
\end{abstract}

\section{Introduction}

Let $\{Y_k\}_{k \in \mathbb{Z}}$ be a stationary sequence of random vectors $Y_k = (Y_{k,1}, \dots, Y_{k,d})'$, $d\in \Z_+$.  
The main interest of this paper is to establish limit theorems for the componentwise maxima over the first $n$ observations of such processes defined as
\begin{equation} \label{eq:sample_extreme} M_{n,i} =  \max_{1\le k\le n} Y_{k,i}, \quad i = 1, \dots, d, \quad n\in \Z_+. \end{equation} 
More specifically,  we  consider for suitable  threshold sequences   $u_n = (u_{n,1}, \dots, u_{n,d})'\in \mathbb{R}^d$,  the limit of 
\begin{equation} \label{eq:sample_extreme_max}   \mathbb{P}\left( M_{n,1} \leq u_{n,1},\ldots, M_{n,d} \leq u_{n,d} \right),   
\end{equation}
as $n\rightarrow\infty$.
This scheme covers the affine normalization in the classical multivariate extreme value theory by setting $u_{n,i}=u_{n,i}(x_i)=x_i/{a_{n,i}}+b_{n,i}$, where $a_{n,i}>0$ and $b_{n,i}\in \R$ are suitable scaling and centering sequences, $i=1,\ldots,d$, and $(x_1,\ldots,x_d)'\in \R^d$ are coordinate variables for the limit joint CDF. More generally,  following  \cite{leadbetter1983extremes}, we shall consider $u_n$ so that each marginal tail probability $\P(Y_{k,i}>u_{n,i})\sim\tau_i/n$, as $n\rightarrow\infty$, for some constant $\tau_i\in (0,\infty)$; see also \eqref{eq:tailbehavemulti_1} below. 

For independent and identically distributed  (i.i.d.) random vectors $\{Y_k\}$, complete descriptions of multivariate extremes can be traced back to the seminal works of  \citet{de1977limit} and \citet{Pickands1981multivariate}. We also refer to the monographs \cite{Resnick1987extreme,deHaanFerreira2006}.
In order to weaken the assumption of independence, 
for univariate sequences,  the limit was studied under certain mixing conditions on $\{Y_k\}$ introduced by \cite{Leadbetter1974on}, usually denoted by $D(u_n)$, where $\pp{u_n}$ denotes a sequence of threshold values that can be related to each $u_{n,i}$ in \eqref{eq:sample_extreme_max}; see  also
\cite{davis1979maxima,davis1982limit}. In addition, to ensure the same asymptotic behavior as the i.i.d.\ case,   an anti-clustering condition denoted as $D'(u_n)$  was made in \cite{leadbetter1983extremes_paper} to restrict the probability of co-exceedances.  See the monograph \cite{leadbetter1983extremes} for a summary of the univariate results.
These conditions and the corresponding limit theorems were later generalized to the multivariate setting by \cite{hsing1989extreme}.  

In the univariate case, to characterize the behavior of the maximum in the presence of strong local dependence, where exceedances tend to cluster so that the limit behavior of extremes may deviate  from the i.i.d.\ case, \cite{leadbetter1983extremes_paper} introduced the extremal index. This parameter can be interpreted as the reciprocal of the expected cluster size of extremes; see also \cite{hsing1988exceedance} for related developments.  \cite{nandagopalan1990multivariate,nandagopalan1994multivariate} extended the notion of extremal index to the multivariate setup; see also \cite{smithcharacterization}. Further important developments, e.g., \cite{davis1995point,basrak2009regularly},   provided   explicit descriptions of a typical cluster under the assumption of joint regular variation, leading to the notion of tail process. A comprehensive survey of these developments can be found in recent monographs \cite{kulik20heavy,mikosch2024extreme}. {A recent related example is \citet{bai2023tail}, where the extremal index and the local clustering
structure of extremes are analyzed through tail-process techniques in a setting where the classical
mixing-based approach is not  adopted.}

The results in the literature on the extremes of stationary processes typically rely on certain mixing conditions (e.g., the aforementioned $D(u_n)$ condition) to control the temporal dependence.  In this work, we shall deal with a setup where it is challenging for these mixing conditions to provide sharp dependence controls:  we consider the so-called Gaussian subordination model obtained by transforming an underlying stationary Gaussian process; see \eqref{eq:gaus sub Y multi} below. In particular, we allow the underlying stationary Gaussian processes to exhibit \emph{long memory} (see, e.g., \cite{pipiras2017long}), which   leads to celebrated examples of \cite{rosenblatt1956central}  where strong mixing type conditions may cease to hold.
This  Gaussian subordination model might be best known in terms of the study of its partial sum behavior (e.g., \cite{dobrushin1979noncentral},  \cite{breuer1983central}).  We note that for stationary Gaussian processes themselves, sharp results on how temporal dependence affects their extreme behaviors are well known and date back to \cite{berman1964limit}; see also \cite{leadbetter1983extremes}. However, results on extremes of general (nonmonotonic) transforms of stationary Gaussian processes are scarce, with the most significant development found in \cite{sly2008nonstandard}. See also \cite{kulik2012limit} for results on a stochastic volatility model involving Gaussian subordination.

From the perspective of modeling,
Gaussian subordination is  very flexible, enabling a wide range of marginal distributions, including heavy-tailed ones, and can easily generate temporal and cross-sectional correlation.  
However, some care is needed when it is used to model extremal dependence due to the special nature of Gaussianity, as explained below. 
It is well known \citep{sibuya1960bivariate} that a bivariate Gaussian $(X_1,X_2)$ with correlation $\rho\in (-1,1)$  is extremally independent, namely, $\lim_{x\rightarrow\infty}\P(X_1>x|X_2>x)=0$. 
Moreover, extremal independence persists for $(f(X_1), g(X_2))$ under arbitrary measurable functions $f,g$ (whenever the notion of extremal independence is well-defined).
This phenomenon also extends to a multivariate setup  where $X_1$ and $X_2$ are replaced by a pair of jointly Gaussian vectors with non-unity canonical correlation.   We refer to Section \ref{sec:gaus exp dep} for a detailed discussion. Furthermore, as a truncation strategy to handle Gaussian subordination with a transform spanning possibly an infinite time horizon,  we introduce the notion of $m$-extremal dependence, a notion analogous to the classical $m$-dependence, which may be of independent interest (see Section \ref{sec:m-extremal-dep}).

As mentioned above, a fundamental concept for quantifying the clustering in extremes in the multivariate setup is the multivariate extremal index, originally proposed in \cite{nandagopalan1990multivariate}. We provide a detailed introduction and formal definition in Section \ref{se:multiexindex}. The multivariate extremal index will allow us to characterize the limit in \eqref{eq:sample_extreme_max} and plays a fundamental role in the formulation of our main result (see Theorem \ref{Thm:limit max MA} below).  Behind this main result is the convergence  of a point process (see Theorem \ref{Thm:gap point proc})    of the form
\begin{equation} \label{eq:pointproces}
\xi_{n}=\sum_{j=1}^\infty \delta_{\pp{  Y_{n}(j) ,\,  a_n(j)}},
\end{equation}
where $Y_{n}(j)$ is a suitable  transformation of data $\{Y_k\}$ and $a_n(j)$ is a suitably normalized time index, both to be explained in detail around \eqref{eq:block conv} in Section \ref{se:main_results}. As usual, $\delta_x$ denotes the Dirac measure at point $x$. 
Our construction of the point process \eqref{eq:pointproces} differs from the usual  ones by considering sample blocks with certain gaps that need to be chosen carefully based on the $m$-extremal dependence of the underlying process.
To prove the point process convergence under Gaussian subordination, we refine and extend the ideas introduced in \cite{sly2008nonstandard} which linked the hypercontractivity property of Gaussian Hilbert spaces to the study of extreme values of Gaussian subordinated processes.  In particular, our assumption allows the covariance of the underlying Gaussian to decay more slowly than any polynomial rate,  whereas \cite{sly2008nonstandard} worked with the typical long memory assumption that stipulates a polynomial rate.

To illustrate our assumption on the underlying Gaussian process, recall the following classical result.
Let $\{X_k\}_{k \in \mathbb{Z}}$ denote a stationary  Gaussian process  with autocovariance function $\gamma(n) = \Cov(X_0, X_n)$. 
It is well known that under the mild decay condition of \cite{berman1964limit}:
\begin{equation} \label{eq:generalcovbehave} \lim_{n \to \infty} \gamma(n) \log(n) = 0, \end{equation} 
the maximum of the process, appropriately normalized in the same way as for i.i.d.\ standard Gaussian, converges weakly to a Gumbel distribution. More generally, the associated point process of exceedances converges weakly to a Poisson point process. 
See \citet[Chapter 4]{leadbetter1983extremes} and    \citet[Theorem 16.2.1]{kulik20heavy}. 
We also note that the intermediate regime
$\lim_{n\to\infty}\gamma(n)\log n\in(0,\infty)$ is studied in \cite{mittal1975limit} for corresponding limit theorems.
Our assumptions in Section \ref{se:model} below, if specialized to the univariate case, first require that $\pc{X_k}$ admits a one-sided moving average representation $X_k=\sum_{l=0}^\infty \psi_l \epsilon_{k-l}$ with centered Gaussian    i.i.d.\ $\pc{\epsilon_l}$. Recall that this is assuming $\{X_k\}$ to be purely nondeterministic, a mild condition characterized by the existence of a spectral density $f(\lambda)$ satisfying the Kolmogorov condition $\int_{-\pi}^\pi \log f(\lambda)d\lambda>-\infty$; see \citet[Section 10.7]{grenander1958toeplitz}.  Note that $\gamma(k)=\sum_{l=0}^\infty \psi_{k+l}\psi_l$ under the assumption.  In addition, we suppose that 
\begin{equation}\label{eq:psi l decay univariate}
\psi_{l}=o\left(l^{-1/2} \log(l)^{-1}\right)
\end{equation}
as $l\rightarrow\infty$. The   condition \eqref{eq:psi l decay univariate} can be shown (see Lemma \ref{le:behavior_acf} below) to imply the Berman condition \eqref{eq:generalcovbehave}. On the other hand, condition \eqref{eq:psi l decay univariate} may be viewed as not much stronger than the Berman condition. For example, it allows the possibility $\lim_{n\rightarrow\infty} \gamma(n)n^{\gamma}=\infty$ for any $\gamma>0$; see, e.g., \citet[Section 2.2.4]{pipiras2017long}.  

Finally, we demonstrate the utility of our results, by deriving a novel limit theorem for the so-called  \emph{multivariate maxima of moving maxima (M4)} process defined as
\begin{equation}\label{eq:gaus sub max MA intro}
{Y^{\mathrm{M4}}_{k,i}}=
\bigvee_{r=-\infty}^\infty \bigvee_{j=1}^d a_{ij,r} W_{k-r,j}
\end{equation}
for some suitable sequence $\{ (a_{ij,r})_{i,j=1,\dots,d}\}_{r\in \Z}$ of nonnegative constants and  $W_{k,i}$ has an identical marginal distribution that is regularly varying. In contrast to typical existing results, we do not  assume independence between $W_{k,i}$'s, but instead, we  assume $\{W_k=(W_{k,1},\ldots,W_{k,d})'\}$ is a subordinated Gaussian process and only impose extremal independence in $W_{k,i}$'s. See \eqref{eq:gaus sub max MA} below for more details.
In the univariate setting, models of this type with i.i.d.\  innovation variables $\{W_k\}$ are frequently considered in the context of extreme value analysis of time series; see, e.g., \cite{hsing1986extreme,davis1989basic,hall2002moving}.
The extremal index of a univariate moving maxima process with independent innovations was studied in \cite{weissman1995extremal}.
The multivariate maxima of moving maxima (M4) models were introduced in \cite{smithcharacterization}; see also
\cite{zhang2004behavior,heffernan2007asymptotically,ferreira2012multivariate,tang2013sparse}.
M4 models with dependent, yet extremally independent, innovation variables, to the best of our knowledge, have not been considered in the literature.  

\subsection{The basic model} \label{se:model}
We introduce a multivariate   subordinated Gaussian process. Let $\{X_k=(X_{k,1},\ldots,X_{k,d_0})'\}_{k \in \Z}$ be a stationary, $d_0$-dimensional, Gaussian process with zero mean marginally, $d_0\in \Z_+$.  We consider the following $(m+1)$-moving-window multivariate Gaussian subordination model $\{Y_k\}_{k \in \Z}$ with $Y_k = (Y_{k,1},\dots, Y_{k,d})'$, $d\in \Z_+$, and 
\begin{equation}\label{eq:gaus sub Y multi}
   Y_{k,i}=G_i(X_k,X_{k-1},\ldots,X_{k-m}), \quad k\in \Z, \quad i =1, \dots, d, 
   \quad \text{ and } \quad m \geq 0,
\end{equation}
where $G_i:\R^{ d_0\times (m+1)}\to \R$ are measurable functions, and $m$ is a fixed nonnegative integer.  Note that the dimension $d$ of $Y_k$ could differ from the dimension $d_0$ of $X_k$.

For our setting, we assume that the latent process $\{ X_k\}$ in \eqref{eq:gaus sub Y multi} can be represented as a multivariate causal linear process
\begin{equation}\label{equality_general_linear_process}
X_{k} = \sum_{l=0}^\infty \Psi_{l} \varepsilon_{k-l},
\end{equation}
where  {$\{\Psi_{l}\}_{l \geq 0}$ is a suitable sequence of matrices with $\Psi_{l}=(\psi_{ij,l})_{i,j=1,\dots,d_0}\in \R ^{d_0 \times d_0}$  and 
$\{\varepsilon_{j}\}_{j\in \Z}$ is a sequence of mean zero i.i.d.\ Gaussian random vectors with identity covariance matrix $\esp(\varepsilon_{0}\varepsilon_{0}')=I_{d_0}$. 
Similar to \eqref{eq:psi l decay univariate}, the entries of the matrices $\{\Psi_{l}\}_{l \in \Z}$ are assumed to satisfy
\begin{equation} \label{equality_long_range_dep_linear_process}
\psi_{ij,l}=o\pp{  l^{-1/2} \log(l)^{-1} }, \hspace*{0.2cm} \text{ as } l \to \infty,
\end{equation}
for $i,j=1,\dots,d_0$.  This assumption implies $\sum_{l=0}^\infty \psi_{ij,l}^2<\infty$, and thus the well-definedness of \eqref{equality_general_linear_process}. In addition, it implies the following multivariate analogue of   \eqref{eq:generalcovbehave}:
\begin{equation} \label{eq:Gammamultiasymp} \Gamma(h) := \mathbb{E}[X_h X_0'] = o\left( \frac{1}{\log(h)} \right), \end{equation}
as $h\rightarrow\infty$
entrywise; see Lemma \ref{le:behavior_acf}. 
Furthermore, we assume that
\begin{equation}\label{eq:nonsingular}
\Psi_0 \text{ is nonsingular}.
\end{equation}
In particular, this implies nonsingularity of
\[
\Gamma(0)=\sum_{\ell=0}^{\infty}\Psi_\ell\Psi_\ell'.
\]

As mentioned before, there is limited literature on extreme value behavior  of subordinated Gaussian processes. A    summary of the existing results can be found in Section 16.2.3 of \cite{kulik20heavy}; see also Section 16.4 therein for a brief review of the literature, including \cite{davis1983stable} and \cite{sly2008nonstandard}. 
As noted in Section 16.2 in \cite{kulik20heavy}, the results in \cite{davis1983stable} implicitly assumed that the function $G$ (in the univariate setup) is monotonic. Neither \cite{sly2008nonstandard} nor our results require such an assumption.

\begin{Rem}
The general model considered in this paper covers a number of familiar classes of transformations. These include, for example, neural-network transforms of a finite window of the input process, of the form
\[
Y_{k,i}
=
\mathrm{NN}_i\!\big((X_k',\dots,X_{k-m}')'\big),
\]
where \(\mathrm{NN}_i:\mathbb{R}^{d_0 \times (m+1)}\to\mathbb{R}\) is the input--output map of a feedforward neural network, that is,
\[
\mathrm{NN}_i(x)
=
W_{i,L}\phi_{i,L-1}\!\Big(
W_{i,L-1}\phi_{i,L-2}\big(
\cdots \phi_{i,1}(W_{i,1}x+b_{i,1})\cdots
\big)+b_{i,L-1}
\Big)+b_{i,L},
\]
for suitable weight matrices \(W_{i,\ell}\), bias vectors \(b_{i,\ell}\), and measurable activation functions \(\phi_{i,\ell}\). We do not treat these specific examples in detail since they are readily contained in the general framework.
\end{Rem}

\subsection{Multivariate extremal index} \label{se:multiexindex}

We first recall the notion of \emph{extremal index} in the univariate setting following Section 3.7 in \cite{leadbetter1983extremes}.
\begin{Def}\label{Def:ext index}
A one-dimensional stationary process $\{Y_k\}$ is said to have an extremal index $\theta\in [0,1]$ if for each $\tau>0$, there is a sequence $\{ u_{n}(\tau)\}$ such that
 \begin{equation}\label{eq:x_n}
 \lim_{n \to \infty} n\P\pp{Y_0> u_{n}(\tau)  }=\tau,
\end{equation}
and
 \begin{equation}\label{eq:ext index exp}
  \lim_{n \to \infty} \P\pp{\max\{Y_1,\ldots,Y_n\}\le u_{n}(\tau)}= e^{-\theta \tau}.
 \end{equation} 
\end{Def}
It is known that \eqref{eq:x_n}  is equivalent to 
\begin{equation}\label{eq:uni exp limit}
  \lim_{n \to \infty} \P\left( \max\Big\{\widehat{Y}_1,\ldots,\widehat{Y}_n\Big\}  \leq u_{n}(\tau)\right)= e^{-\tau},
\end{equation}
where  $\{\widehat{Y}_n \}_{n\in \Z}$   are i.i.d.\ copies of ${Y}_0$; see  \citet[Theorem 1.5.1]{leadbetter1983extremes}. 
In other words, comparing \eqref{eq:uni exp limit} to \eqref{eq:ext index exp}, the parameter $\theta$ adjusts the limit in the i.i.d.\ case to account for temporal dependence. In particular, $\theta=1$ suggests the same behavior as in the i.i.d.\ case, while $\theta<1$ signals the temporal clustering of extremes. The reciprocal of $\theta$ often bears the interpretation of average extremal cluster size; see    \citet[Chapter 6]{mikosch2024extreme} for more details.
We also mention that the existence of $u_n(\tau)$  such that the limit relation \eqref{eq:x_n} holds is equivalent to 
\begin{equation}\label{eq:upper end reg}
\lim_{x\rightarrow U_F}\frac{\widebar{F}(x)}{\widebar{F}(x-)}=1,
\end{equation}
where $\widebar{F}(x)=\P(Y_0>x)$, $U_F=\sup\{x: \widebar{F}(u)>0\}$, and $\widebar{F}(x-)$ denotes the left limit of $\widebar{F}$ at $x$; 
see \citet[Theorem 1.7.13]{leadbetter1983extremes}. The condition \eqref{eq:upper end reg} is satisfied if $\widebar{F}$ is continuous {in a left neighborhood of $U_F$}. It is also known to hold if the law of $Y_0$ is in the max-domain of attraction of an extreme value distribution;   see the proof of \citet[Theorem 1.6.2]{leadbetter1983extremes}.

The \emph{multivariate extremal index} extends the concept of the one-dimensional extremal index to cases where $Y_k$ is vector-valued, whose definition goes back to \cite{nandagopalan1990multivariate,nandagopalan1994multivariate}; see also \cite{smithcharacterization,martins2005multivariate,robert2008estimating}.
Following \cite{martins2005multivariate}, a formal definition is given next.
\begin{Def} \label{Def:ext index multi}
Let $\{Y_n\}_{n\in\Z}$ be a  stationary sequence taking values in $\R^d$. 
For each $\tau=(\tau_1,\dots,\tau_d)'$, $\tau_i\in (0,\infty)$, suppose there exists a threshold sequence $u_n(\tau)=(u_{n,1}(\tau_1),\dots,u_{n,d}(\tau_d))'$ satisfying the marginal tail condition 
\begin{equation} \label{eq:tailbehavemulti_1}
    n \P( Y_{0,i} > u_{n,i}(\tau_i)) \to \tau_i, \quad i =1, \dots, d,
\end{equation}
as $n\to\infty$. 
{
Write
\begin{equation}\label{eq:M_n component}
M_n=(M_{n,1},\dots,M_{n,d})',
\qquad
M_{n,i}=\max_{1\le k\le n} Y_{k,i},\quad i=1,\dots,d,
\end{equation}
for the vector of componentwise maxima, and let
\[
\widehat{M}_n=(\widehat{M}_{n,1},\dots,\widehat{M}_{n,d})',
\qquad
\widehat{M}_{n,i}=\max_{1\le k\le n}  \widehat{Y}_{k,i},\quad i=1,\dots,d,
\]
where $\{ \widehat{Y}_k\}_{k\in\mathbb Z}$ is an i.i.d.\ sequence of copies of $Y_0$.
}
The multivariate extremal index $\theta(\tau)\in[0,1]$ is then defined by the following limit relations as
$n\to\infty$, 
\begin{equation} \label{eq:tailbehavemulti_2}
    \P( \widehat{M}_n \leq u_{n}(\tau)) \to {H}(\tau),\  
    \text{ for some        function } {H}(\tau)\in (0,1),\ \tau\in (0,\infty)^d,
\end{equation}
and
\begin{equation} \label{eq:tailbehavemulti_3}
    \P( M_n \leq u_{n}(\tau)) \to {H}(\tau)^{\theta(\tau)}.
\end{equation}
Here and throughout, the inequality $\le$ between two vectors (e.g., $M_n \leq u_{n}(\tau)$) of the same length is understood   componentwise, and $\nleq$ means the opposite, i.e., at least one component on the left is strictly greater than the corresponding component on the right.
\end{Def}

\begin{Rem}
We note that unlike the univariate case,   the relations \eqref{eq:tailbehavemulti_1} and \eqref{eq:tailbehavemulti_2} are generally not equivalent as the former only involves the marginal distributions. On the other hand,  applying the inequality $e^{-nx/(1-x) }\le (1-x)^n\le e^{-nx}$, $x\in [0,1)$, $n\in \Z_+$,  to   $\P( \widehat{M}_n \leq u_{n}(\tau))=\pb{1-\P\pp{Y_0\nleq u_n(\tau)}}^n$,  and  taking into account the relation   $\max_{1\le k\le n}\P\pp{Y_{0,i}>u_{n,i}(\tau_i)} \le\P\pp{Y_0\nleq u_n(\tau)}\le \sum_{i=1}^d \P(Y_{0,i}>u_{n,i}(\tau_i))$ as well,  
we see that \eqref{eq:tailbehavemulti_1} implies
\[
e^{-(\tau_1+\ldots+\tau_d)}\le \liminf_{n \to \infty} \P( \widehat{M}_n \leq u_{n}(\tau)) \le \limsup_{n \to \infty} \P( \widehat{M}_n \leq u_{n}(\tau))\le e^{-\max(\tau_1,\ldots,\tau_d)}.
\]
In addition, under \eqref{eq:tailbehavemulti_1},   the limit relation \eqref{eq:tailbehavemulti_2} is equivalent to  the limit relation
\begin{equation}\label{eq:limit G(tau)}
       \lim_{n\rightarrow\infty} n \P\pp{Y_0\nleq u_n(\tau)}= -\log\pc{ {H}(\tau) }, \  
    \text{ for some        function } {H}(\tau)\in (0,1),\ \tau\in (0,\infty)^d,
\end{equation}
where if either limit relation holds, we have 
\begin{equation}\label{eq:G bound}
e^{-(\tau_1+\ldots+\tau_d)}\le {H}(\tau)\le e^{-\max(\tau_1,\ldots,\tau_d)}.
\end{equation}
Furthermore, it is known (e.g., \cite{perfekt1997extreme}) that the existence of $u_n(\tau)$ for both \eqref{eq:tailbehavemulti_1} and \eqref{eq:limit G(tau)} to hold follows if $Y_0$ is in the max-domain of attraction of a multivariate extreme value distribution.
\end{Rem}

{
\begin{Rem}
The multivariate extremal index determines the univariate extremal index of each coordinate
process. Indeed, the extremal index of $\{Y_{k,i}\}_{k\in\mathbb Z}$ is recovered by restricting
the multivariate limit relations to the $i$th coordinate axis, equivalently by letting
$\tau_j\downarrow 0$ for all $j\neq i$ {(see, e.g., \cite[Eq.\ (1.6)]{martins2005multivariate})}. In general, however, the function $\theta(\tau)$ is not
determined by the collection of coordinatewise extremal indices alone, since it also reflects
cross-sectional extremal interaction.
\end{Rem}
}

\section{Main results} \label{se:main_results}

{\subsection{General limit theorems for finite-dimensional Gaussian subordination}}

Our main result concerns establishing \eqref{eq:tailbehavemulti_3} for the multivariate subordinated Gaussian process $\{Y_k\}$ in \eqref{eq:gaus sub Y multi}.  
{
Recall from Definition \ref{Def:ext index multi} that for $\tau=(\tau_1,\dots,\tau_d)'\in(0,\infty)^d$,
\[
u_n(\tau)=\bigl(u_{n,1}(\tau_1),\dots,u_{n,d}(\tau_d)\bigr)'
\]
denotes the threshold vector, with each $u_{n,i}(\tau_i)$ corresponding to the marginal
exceedance level indexed by $\tau_i$. The multivariate extremal index $\theta(\tau)$ is defined
relative to this threshold vector. Assume also the limit relation \eqref{eq:limit G(tau)} (or equivalently, \eqref{eq:tailbehavemulti_2}).
}

An approach that goes back to \cite{o1987extreme}  in the univariate setting and was later generalized to multivariate processes in \cite{smithcharacterization}, characterizes the extremal index as the limit of the conditional probability 
\begin{equation}\label{eq:ext index m-dep multi-1}
\theta(\tau)=\lim_{n \to \infty}  \P\pp{Y_{1}\le u_n(\tau),\ldots,Y_{\ell_n}\le u_n(\tau) \mid Y_{0}\nleq u_{n}(\tau)}
\end{equation}
for some  sequence $\{\ell_n\}_{n \geq 1}$ satisfying $\ell_n\rightarrow \infty$ and $\ell_n/n\rightarrow 0$, as $n\rightarrow\infty$.
The limit in \eqref{eq:ext index m-dep multi-1} is known to exist under a certain mixing condition  with respect to the sequence $(\ell_n)_{n \geq 1}$; see Lemma 2.2 in \cite{smithcharacterization}.

We shall suppose that the following limits
\begin{equation}\label{eq:ext index m-dep multi}
\theta_\ell(\tau):=\lim_{n \to \infty}  \P\pp{Y_{1}\le u_n(\tau),\ldots,Y_{\ell}\le u_n(\tau) \mid Y_{0}\nleq u_{n}(\tau)} 
\end{equation}
exist  for all $\ell=0,\dots,m$. Since this involves a finite fixed $m$, it is a mild assumption, which holds if the joint law of $\pp{Y_{0}, Y_{1},\ldots,Y_{m}}$ is in the max-domain of attraction of a multivariate extreme value distribution (see Proposition \ref{Pro:suff MEV} below). It turns out that for the model \eqref{eq:gaus sub Y multi} which is $m$-extremally-dependent (see Section \ref{sec:m-extremal-dep}),  we have $\theta_\ell(\tau)=\theta_m(\tau)=\theta(\tau)$  for all $\ell>m$.  

\begin{Thm}\label{Thm:max}
Suppose $\{Y_k\}_{k\in \Z}$ is as   in \eqref{eq:gaus sub Y multi} with the underlying Gaussian $\{X_k\}_{k \in \Z}$ given as in \eqref{equality_general_linear_process} satisfying \eqref{equality_long_range_dep_linear_process} and \eqref{eq:nonsingular}. Assume for a sequence $\{u_n(\tau)\}_{n \geq 1}$ the limit relations \eqref{eq:tailbehavemulti_1}, \eqref{eq:limit G(tau)}   and \eqref{eq:ext index m-dep multi} hold.
Then, the limit relation 
\begin{equation} \label{eq:sample_extreme_max_res}
    \lim_{n \to \infty} \P\left( M_{n} \leq u_{n}(\tau)\right) 
    =
    {H}(\tau)^{\theta(\tau)}
\end{equation}
holds for the same   ${H}(\tau)\in (0,1)$ for $\tau\in (0,\infty)^d$ as in \eqref{eq:limit G(tau)}, and with extremal index $\theta(\tau)=\theta_m(\tau)$ given in \eqref{eq:ext index m-dep multi}.
\end{Thm}

{
\begin{Rem}
In the univariate case $d=1$, the quantity $\theta(\tau)=\theta_m(\tau)$ in Theorem 2.1
coincides with the ordinary extremal index of $\{Y_k\}$ and is therefore independent of $\tau$.
For $\ell<m$, the quantities $\theta_\ell(\tau)$ may still depend on $\ell$ and should be
viewed as finite-horizon cluster-survival probabilities. Under $m$-extremal dependence,
Proposition \ref{Pro:m extre dep} below shows that these quantities stabilize at lag $m$, namely
$\theta_\ell(\tau)=\theta_m(\tau)$ for all $\ell>m$.
\end{Rem}
}

The proof of Theorem \ref{Thm:max}, given in Section \ref{sec:pf Thm max}, is based on a point process approach. In what follows, we formulate a   general point process result, extending beyond the needs of Theorem \ref{Thm:max}, which may be of independent interest.
 
Let $r\in \Z_+$, and consider  $\pp{ Y_{(j-1)(r+p)+1},\ldots,  Y_{jr+(j-1)p}}$, $j\in \Z$, that is, blocks of consecutive  $Y_k$'s, each of length $r$, with a gap of size $p\ge 0$  between two adjacent blocks.  Assume throughout that 
$$p\ge m,
$$
where $m$ is as in \eqref{eq:gaus sub Y multi}. When $m=0$, that is, when the transform $G_i$ in \eqref{eq:gaus sub Y multi} depends only on a single time coordinate, the gap $p$ can be taken as $0$, as considered by \cite{sly2008nonstandard}.  The requirement $p\ge m$ ensures that the underlying Gaussian $X_k$'s involved in different  $Y$ blocks above are disjoint.

We will use the notion of vague convergence for locally finite measures on a localized Polish space, for which we refer to \citet[Appendix B]{kulik20heavy} for the details.
Suppose $E$ is a localized Polish space.  
Let $T_n:\R^{d \times r} \to  E$, $n\in \Z_+$, be a sequence of measurable transforms and set
\begin{equation}\label{eq:Y_n(j)}
Y_n(j)=Y_{r,p,n}(j)=  T_n \pp{Y_{(j-1)(r+p)+1} ,\ldots,     Y_{jr+(j-1)p}}\in E, \quad  j\in \Z_+,
\end{equation}
where we have suppressed the dependence on $r,p$ in the notation $Y_{n}(j)$; here and below, we put a block-level time index  (e.g., $j$ in $Y_{n}(j)$)   in  parentheses. 
We assume further that there exists a locally finite measure $\Lambda_{r,p}$ on $E$  such that
\begin{equation}\label{eq:block conv}
\frac{n}{r+p} \P \pp{Y_{r,p,n}(1) \in \cdot   } \overset{v}{\rightarrow} \Lambda_{r,p}(\cdot), 
\end{equation}
as $n\rightarrow\infty$, where  $\overset{v}{\rightarrow} $ stands for vague convergence.
Define the point process on $E\times [0,\infty)$ as  
\begin{equation} \label{eq:pointproces-Y_n(j)}
\xi_{n}=\xi_n(r,p)=\sum_{j=1}^\infty \delta_{\pp{  Y_{n}(j) ,\,  j   (r+p)/n}},
\end{equation}
where $\delta$ denotes the Dirac delta measure.

\begin{Thm}\label{Thm:gap point proc}
Suppose $\{Y_k\}_{k \in \Z}$ is as in Theorem \ref{Thm:max}. 
Assume in addition \eqref{eq:block conv} holds. Then, the process $\xi_n$ in \eqref{eq:pointproces-Y_n(j)} satisfies the weak convergence
\[
\xi_{n} \Rightarrow \xi , \text{ as } n\rightarrow\infty,
\]
where $\xi$ is a Poisson point process on $E\times [0,\infty)$ with mean measure $\mu_r=\Lambda_{r,p}\times \Leb$,  and $\Lambda_{r,p}$ as in \eqref{eq:block conv}.
\end{Thm}
The proof of Theorem \ref{Thm:gap point proc} builds on a refinement and extension of the arguments in \cite{sly2008nonstandard}; full details are provided in Section \ref{sec:pf Thm:gap point proc}. We also note that, when applying Theorem \ref{Thm:gap point proc} to establish Theorem \ref{Thm:max}, the effect of the gaps is negligible at the time scale $n$.

{
\subsection{An application to multivariate moving maxima processes}
}

{
We now illustrate the general theory on a multivariate moving maxima (M4) model. The proof of
Theorem \ref{Thm:gap point proc} does not proceed via a direct point-process convergence for the full M4 process.
Instead, we first apply the finite-dimensional theory in the previous section to truncated versions of the process and
then pass to the limit by means of the approximation argument developed later in Lemma \ref{Lem:approx}.
}

Next, we discuss how to extend the results on Gaussian subordinated $\{Y_k\}$ defined by a finite-dimensional transform, as required in Theorems \ref{Thm:max} and \ref{Thm:gap point proc}, to an infinite-dimensional transform. Theorem \ref{Thm:gap point proc} cannot directly handle the case where the function $G$ depends on infinitely many time coordinates of the Gaussian linear process $\{X_k\}$ in \eqref{equality_general_linear_process}. 
However, in certain scenarios, an approximation strategy originating from   \cite{chernick1991calculating} can be adopted (see Lemma \ref{Lem:approx} below). We shall showcase the use of such a strategy on a multivariate moving-maximum built from Gaussian subordinated innovations below.

Let $h:\R \to \R$ be a measurable function and suppose $W_{k,i}=h(X_{k,i})$ with $ X_k=(X_{k,1},\ldots,X_{k,d})'$ a multivariate Gaussian linear process given by \eqref{equality_general_linear_process} (setting $d_0=d\in \Z_+$) satisfying \eqref{equality_long_range_dep_linear_process} with nonsingular covariance \eqref{eq:nonsingular} at time $0$. In fact, this construction 
ensures {extremal} {in}dependence of $\{W_{k,i}\}$, both in time index $k$, and in space index $i$; see Section \ref{sec:gaus exp dep} and Lemma \ref{le:multiBrochwellDavis} below. Note that $W_{k,i}$'s are marginally identically distributed. In fact, it is possible to extend the setup to allow the function $h$ to vary with respect to $k$ and $i$, although for simplicity we do not pursue this here.

We also assume  that the function $h$ is chosen so that $W_{k,i} \in [0,\infty)$ almost surely, and in addition,  
\begin{equation} \label{eq:alphavarying}
    \P( W_{k,i} > u ) =  u^{-\alpha} L(u),\quad u>0,
\end{equation}
 where $\alpha \in (0,\infty)$, and $L(\cdot)$ is slowly varying at infinity.
 The relation \eqref{eq:alphavarying} ensures that the marginal distribution of each $W_{k,i}$ is in the max-domain of attraction of an $\alpha$-Fr\'echet distribution. 
 
Suppose $\{a_{r}\}_{r \in \Z}$, $a_{r} = (a_{ij,r})_{i,j=1,\dots,d}$, is a suitable sequence of real matrices with nonnegative entries. Then, the multivariate maxima of moving maxima (M4)
  process is defined as
\begin{equation}\label{eq:gaus sub max MA}
{ Y_k^{\mathrm{M4}} = (Y_{k,1},\dots, Y_{k,d})'},\quad  
Y_{k,i}^{\mathrm{M4}}=G_{i}((X_{k-l})_{\ell \in \Z})
:= 
\bigvee_{{r\in \Z}} \bigvee_{j=1}^d a_{ij,r} W_{k-r,j}.
\end{equation}
To avoid degenerate marginal distribution of $Y_{k,i}^{\mathrm{M4}}$, we assume for each $i$, we have $a_{ij,r}>0$ for some $j$ and $r$.
An additional assumption  is that for some $\epsilon\in (0,\alpha)$,
\begin{equation} \label{eq:cond1-maxsum}
    \sum_{r \in \Z}^\infty    a_{ij,r}^{\alpha-\epsilon}  <\infty,\quad i,j=1,\dots,d.
\end{equation}
These assumptions imply that $ \sum_{r \in \Z} \sum_{j = 1}^d  a_{ij,r}^{\alpha}\in (0,\infty) $ for any $i=1,\ldots,d$, and that $Y_{k,i}<\infty$ a.s. (see Lemma \ref{Lem:MA tail RV} below). Set 
\begin{equation}\label{eq:A_i}
    A_i:= \sum_{r\in \Z}\sum_{j=1}^d a_{ij,r}^\alpha\in (0,\infty).
\end{equation}
We note that \eqref{eq:gaus sub max MA} is not stated in the most general form of the M4 process described in \cite{smithcharacterization}, where $d$ is allowed to be infinite. In this work we restrict the attention to a finite $d$, the extension to infinite-dimensional $d$ is left for future investigation.

\begin{Thm}\label{Thm:limit max MA}
Let $\{Y_k^{\MF}\}_{k \in \Z}$ be given by \eqref{eq:gaus sub max MA} with $W_{k,i}=h(X_{k,i})$ satisfying \eqref{eq:alphavarying}. 
Under the conditions \eqref{equality_long_range_dep_linear_process}, \eqref{eq:nonsingular}, \eqref{eq:tailbehavemulti_1}, \eqref{eq:cond1-maxsum}, we have
\begin{equation} \label{eq:EVT max MA}
    \lim_{n \to \infty} \P\left( M_{n} \leq u_{n}(\tau) \right) 
    =
    {H}(\tau)^{\theta(\tau)} {=\exp\pc{  -\sum_{j = 1}^d  \bigvee_{r \in \Z} \pp{  \bigvee_{i = 1}^d  \frac{a_{ij,r}^\alpha \tau_i}{A_i} } }},
\end{equation}
where $M_n$ is as in \eqref{eq:M_n component} defined with $Y_{k,i}^{\MF}$,
\begin{equation}\label{eq:limit G(tau) max MA}
    {H}(\tau)=  \exp\pc{ -  \sum_{j =1}^d \sum_{ r\in \Z} \pp{ \bigvee_{i = 1}^d  \frac{a_{ij,r}^\alpha \tau_i}{A_i}}},
\end{equation}
and the multivariate extremal index
\begin{equation} \label{eq:move-max-extremal-index}
    \theta(\tau)
    =
    \frac{
    \sum_{j = 1}^d  \bigvee_{r \in \Z} \left( \bigvee_{i = 1}^d  a_{ij,r}^\alpha \tau_i/A_i\right)}{
    \sum_{j = 1}^d  \sum_{r \in \Z} \left( \bigvee_{i = 1}^d  a_{ij,r}^\alpha \tau_i/A_i\right)}.
\end{equation}
\end{Thm}

{
For the M4 model in Theorem \ref{Thm:limit max MA}, the thresholds can be made explicit, {yielding the following alternative formulation in terms of a conventional multivariate extreme-value limit result.}}
Let $(a_n)$ be such that as $n\rightarrow\infty$
\[
n\,\mathbb{P}(W>a_n)\to 1.
\]
Since {in view of Lemma \ref{Lem:MA tail RV} below}, we have as $u\rightarrow\infty$,
\[
\mathbb{P}(Y_{0,i}^{\mathrm{M4}}>u)\sim A_i u^{-\alpha}L(u),
\qquad
A_i=\sum_{r\in\mathbb{Z}}\sum_{j=1}^d a_{ij,r}^{\alpha},
\]
it follows that one may choose in Theorem \ref{Thm:limit max MA}
\[
u_{n,i}(\tau_i)=A_i^{1/\alpha}a_n x_i,\quad x_i>0, \quad  \tau_i=x_i^{-\alpha}.
\]
Hence the conclusion of Theorem \ref{Thm:limit max MA} may be written in the equivalent multivariate Fr\'echet limit form
\begin{equation}\label{eq:mult frechet}
\mathbb{P}\!\left(
\frac{M_{n,1}}{A_1^{1/\alpha}a_n}\le x_1,\dots,
\frac{M_{n,d}}{A_d^{1/\alpha}a_n}\le x_d
\right)
\to
\exp\{-V(x_1,\dots,x_d)\},
\end{equation}
where the exponent function $V$ is defined by 
\begin{equation}\label{eq:exponent function V}
    V(x_1,\ldots,x_d)= \sum_{j = 1}^d  \bigvee_{r \in \Z} \pp{  \bigvee_{i = 1}^d  \frac{a_{ij,r}^\alpha x_i^{-\alpha}}{A_i} }
\end{equation}
}

{
\begin{example}
Let $d=2$ and suppose the only nonzero coefficients are
\[
a_{11,0}=a_{11,1}=1,
\qquad
a_{21,0}=1,
\]
with all remaining $a_{ij,r}=0$. Then
\[
Y_{k,1}=W_k\vee W_{k-1},
\qquad
Y_{k,2}=W_k.
\]
Hence
\[
A_1=2,
\qquad
A_2=1.
\]
For the first coordinate,
\[
\theta_1
=
\frac{\sum_{j=1}^d \max_r a_{1j,r}^{\alpha}}
{\sum_r\sum_{j=1}^d a_{1j,r}^{\alpha}}
=
\frac{1}{2},
\]
while for the second coordinate,
\[
\theta_2
=
\frac{\sum_{j=1}^d \max_r a_{2j,r}^{\alpha}}
{\sum_r\sum_{j=1}^d a_{2j,r}^{\alpha}}
=
1.
\]
Thus the two coordinates may have different univariate extremal indices, even though they are
driven by the same innovation sequence.
Moreover, in this example
\[
{H}(\tau_1,\tau_2)
=
\exp\{-(\tau_1+\tau_2)\},
\]
while
\[
\theta(\tau_1,\tau_2)
=
\frac{\max(\tau_1/2,\tau_2)+\tau_1/2}{\tau_1+\tau_2}.
\]
This shows that the multivariate extremal index depends on the direction $\tau$ and is not
determined by $(\theta_1,\theta_2)$ alone.
\end{example}
}

\begin{Rem}
The exponent function associated with \(H(\tau)\) in \eqref{eq:limit G(tau) max MA} is that of a spectrally discrete
max-stable random vector {(see, e.g., \cite{StrokorbSchlather2015})}. Indeed, after marginal standardization, the corresponding max-stable
vector admits a max-linear representation with spectral weights determined by the coefficients
\(a_{ij,r}\). More precisely, let \((Z_{r,j})_{r,j}\) be i.i.d.\ standard \(\alpha\)-Fr\'echet random variables, and set
\[
\widetilde M_i
=
\bigvee_{r\in\mathbb{Z}}\bigvee_{j=1}^d
\frac{a_{ij,r}}{A_i^{1/\alpha}}\, Z_{r,j},
\qquad i=1,\dots,d,
\]
Then the multivariate $\alpha$-Fr\'echet random vector $\wt{M}=\pp{\wt{M}_1,\ldots,\wt{M}_d}'$ has an exponent function corresponding to \(H\).

The clustered limit \(H(\tau)^{\theta(\tau)}\) in \eqref{eq:limit G(tau) max MA} also admits a spectrally discrete max-stable representation. Indeed, writing
\[
\bar a_{ij}:=\bigvee_{r\in\mathbb Z} a_{ij,r},
\qquad i,j=1,\dots,d,
\]
and letting \((\bar Z_j)_{j=1}^d\) be i.i.d.\ standard \(\alpha\)-Fr\'echet random variables, define
\[
  M_i^{\mathrm{cl}}
=
\bigvee_{j=1}^d
\frac{\bar a_{ij}}{A_i^{1/\alpha}}\,\bar Z_j,
\qquad i=1,\dots,d.
\]
Then \(  M^{\mathrm{cl}}=\pp{  M_1^{\mathrm{cl}},\dots,  M_d^{\mathrm{cl}}}'\) has exponent function $V$
 precisely as in \eqref{eq:exponent function V}.
\end{Rem}

\section{Preliminaries}
In this section, we collect some preliminary results that are key to proving our main results and may be of independent interest. In particular, 
we provide some technical background on Gaussian analysis in Section \ref{se:correlation-hyper}, discuss the extremal dependence for Gaussian subordination in Section \ref{sec:gaus exp dep} and introduce the notion of extremal $m$-dependence in Section \ref{sec:m-extremal-dep}. Finally, Section \ref{se:linear process properties} discusses conditions on the linear process representation of the latent Gaussian process.

\subsection{Correlation, projection and hypercontractivity} \label{se:correlation-hyper}

We review certain results about Gaussian spaces that will play key roles in our analysis. Our main reference is \citet{janson1997gaussian}.  

Throughout we fix an underlying probability space, and use  $L^p$ to denote the $L^p$-space of random variables defined on the probability space equipped with the $L^p$-norm $\|\cdot \|_p$, $p\in (0,\infty]$.  For a finite-dimensional real vector, we always use $\|\cdot\|$ to denote its Euclidean norm.  
Recall that a (real) Gaussian Hilbert space  $H$ is  a closed subspace of the Hilbert space $L^2$  consisting of (real) centered Gaussian random variables. We use $L^p(H)$ to denote the  subspace of $L^p$ consisting  of random variables measurable with respect to the $\sigma$-field  $\sigma(H)$ generated by the random variables in $H$.

Suppose  $H_1$ and $H_2$  are two Gaussian Hilbert spaces which are closed subspaces of  $H$. The \emph{canonical correlation} $r(H_1,H_2)$ is defined as  
\begin{equation}\label{eq:can corr}
r(H_1,H_2)=\sup_{X_1\in H_1, X_2\in H_2} \mathrm{Corr}(X_1,X_2)=\sup_{X_1\in H_1, X_2\in H_2} \frac{\Cov(X_1,X_2)}{\|X_1\|_2 \|X_2\|_2},
\end{equation}
and the \emph{maximal correlation} $\rho_{HK}$ is defined as
\begin{equation}\label{eq:max corr}
\rho(H_1,H_2)=\sup_{X_1\in L^2(H_1), X_2\in L^2(H_2)} \mathrm{Corr}(X_1,X_2)=\sup_{X_1\in L^2(H_1), X_2\in L^2(H_2)} \frac{\Cov(X_1,X_2)}{\|X_1\|_2 \|X_2\|_2}.
\end{equation}
In both \eqref{eq:can corr} and \eqref{eq:max corr}, if $\|X_1\|_2$ or $\|X_2\|_2=0$, we understand the correlation or the ratio as $0$.  
Next,
let $P_{H_1H_2}:H_1\to H_2$ denote the orthogonal projection (from $H$) onto $H_2$ restricted to the domain $H_1$, which can be identified with $X\to \esp [X \mid \sigma(H_2)]$. Consider also the operator
\begin{equation}\label{def:QH1H2}
    Q_{H_1H_2}:  L^2(H_1) \to L^2(H_2)  ,  \ X\to \esp [X \mid \sigma(H_2)]-\esp [X].
\end{equation}
The operator norms (with respect to the $L^2$ norm) of $P_{H_1H_2}$  and $Q_{H_1H_2}$ are denoted as $\|P_{H_1H_2}\|$ and $\|Q_{H_1H_2}\|$, respectively.   
The following result is known which identifies the aforementioned correlations and norms; see \citet[Theorem 10.11  and its proof]{janson1997gaussian} and \citet[Theorem 1]{kolmogorov1960strong}.
\begin{Pro}\label{Pro:rho equal}
We have $r(H_1,H_2)=\rho(H_1,H_2)=\|P_{H_1H_2}\|=\|Q_{H_1H_2}\|$.
\end{Pro}

Recall that the $k$-th Wiener chaos is defined to be $H^{:k:}=\widebar{\cl{P}}_{k}\cap \widebar{\cl{P}}_{k-1}^{\perp}$ with $\widebar{\cl{P}}_{k}$ denoting the closure in $L^2$ of the space of (multivariate) polynomials of elements in $H$ each with a degree not exceeding $k$, $k\in \Z_+$. The $0$th order chaos $H^{:0:}$ is understood as the space of constant random variables. The  spaces $H^{:k:}$ of different order $k$'s are orthogonal to each other, and for any $X\in L^2(H)$, we have  the Wiener chaos expansion {\cite[Theorem 2.6]{janson1997gaussian}}  in $L^2$ as 
\[
X=\sum_{k=0}^\infty \pi_k(X),
\]
where $\pi_k:L^2(H)\to H^{:k:}$ stands for the projection onto $H^{:k:}$.  Note that $\pi_0 X=\esp X$.
The Mehler transform $M_a:L^2(H)\to L^2(H)$, $a\in [-1,1]$, is defined as 
\begin{equation}\label{eq:M_a}
    M_a(X)=\sum_{k=0}^\infty a^k \pi_k(X).
\end{equation}
 By the orthogonality, we know   $\Var\pp{M_a(X)} = \sum_{k=1}^\infty a^{2k} \|\pi_k(X)\|_2^2$, and hence the following monotonicity property holds 
\begin{equation}\label{eq:Mehler mono}
    \Var\pp{M_{a_1}(X)}\le \frac{a_1^2}{a_2^2} \Var\pp{M_{a_2}(X)}\le \Var\pp{M_{a_2}(X)} , \quad \text{for }0\le  |a_1|< |a_2|\le 1.
\end{equation}
Furthermore, we recall the following hypercontractivity property   \citep[Theorem 5.8]{janson1997gaussian}
\begin{equation}\label{eq:M_r hypercontr}
 \|M_a(X)\|_2\le \|X\|_{1+a^2}.    
\end{equation}

Below, we still take $H_1$ and $H_2$ to be Gaussian Hilbert subspaces of  $H$ and write $\rho=\rho(H_1,H_2)$.
The following result provides a relation between the operators $Q_{H_1H_2}$ and $M_a$.
\begin{Pro}\label{Pro:proj mehler}
Recall $Q_{H_1H_2}$ from \eqref{def:QH1H2} and $M_\rho$ from \eqref{eq:M_a}. Then, for any $X\in L^2(H_1)$, we have
\[
\|Q_{H_1H_2}(X)\|_2^2 \le \Var\pp{M_\rho(X)}.
\]
\end{Pro}
\begin{proof}
By \citet[Theorem 4.9]{janson1997gaussian},   we have
\[
Q_{H_1H_2}(X) =\sum_{k=1}^\infty P_{H_1H_2}^{:k:} \pi_k\pp{X},
\]
where $P_{H_1H_2}^{:k:} : H^{:k:}\to H^{:k:}$ is the operator as in \citet[Theorem 4.5]{janson1997gaussian}, i.e., the $k$th symmetric tensor power of $P_{H_1H_2}$. In view of \citet[Theorem 4.5]{janson1997gaussian} and Proposition \ref{Pro:rho equal}, the operator norm  $\|P_{H_1H_2}^{:k:}\|=\|P_{H_1H_2}\|^k=\rho^k$. Hence, by orthogonality of chaoses of different orders,
\[
\|Q_{H_1H_2}(X)\|_2^2\le \sum_{k=1}^\infty \rho^{2k} \|\pi_k\pp{X}\|_2^2 = \| M_\rho(X) -\esp M_\rho(X)\|_2^2 = \Var(M_\rho(X)).
\]
\end{proof}
 The following  improvement of H\"older's inequality based on the  hypercontractivity property    will play a key role in analyzing extremal dependence in Section \ref{sec:gaus exp dep} below. 
 
\begin{Pro}\label{Pro:hyper holder}
 For random variables $X_1$ and $X_2$ measurable with respect to $\sigma(H_1)$ and $\sigma(H_2)$ respectively, and $p,q\ge 1$ with $(p-1)(q-1)\ge  \rho^2$, where $\rho$ is the canonical correlation between $H_1$ and $H_2$, we have
 \[
 \esp \left| X_1 X_2  \right|\le \|X_1\|_p \|X_2\|_{q}.
 \]
 In particular, with the choice $p=q=1+\rho$, we have
 \begin{equation}\label{eq:hyper CS}
 \esp \left| X_1 X_2  \right|\le \|X_1\|_{1+\rho}\|X_2\|_{1+\rho}.
 \end{equation}
\end{Pro}
 The proof of Proposition \ref{Pro:hyper holder} follows from a direct generalization of that of \citet[Corollary 5.7]{janson1997gaussian} using Proposition \ref{Pro:rho equal}.

\subsection{Gaussian Subordination and Extremal Independence}\label{sec:gaus exp dep}

We discuss extremal independence for Gaussian subordination.  

Suppose $(Y_1,Y_2)'$ is a bivariate random vector with an identical  marginal distribution function $F(x)=\P\pp{Y_1\le x}=\P\pp{Y_2\le x}$, $x\in \R$.  Recall the upper end point of $F$ as $U_F=\sup\{x\in \R : \  F(x)<1 \}$. Assume $\lim_{x\uparrow U_F}F(x)=1$, that is, the marginal distribution is continuous at $U_F$. 
Then $Y_1$ and $Y_2$ are said to be \emph{extremally   independent} (or say asymptotically independent, or having zero tail dependence, in the upper tail), if  
\begin{equation}\label{eq:extr indep}
\lim_{x\uparrow U_F}\P(Y_1>x ~|~ Y_2>x)=\lim_{x\uparrow U_F}\frac{\P \pp{Y_1>x,\ Y_2>x}}{\widebar{F}(x)}    =0,
\end{equation}
where $\widebar{F}(x)=1-F(x)=\P \pp{Y_1>x}$.
Furthermore, we say that two random vectors, all components with an identical marginal {distribution} that is continuous at its upper end point, are extremally independent if pairwise extremal independence holds across the components of the two vectors.

Suppose now $(Y_1,Y_2)$ is bivariate normal with correlation $\rho\in (-1,1)$.  Then it is   known  that
\begin{equation}\label{eq:normal extr indep rate}
\P(Y_1>x, Y_2>x)\sim   c_\rho  \log\pp{1/\widebar{F}(x) }^{-\rho/(1+\rho)}  \widebar{F}(x)^{2/( 1+\rho )},
\end{equation}
as $x\rightarrow\infty$, for some constant $c_\rho>0$ depending on $\rho$. The relation \eqref{eq:normal extr indep rate} is a variant of \citet[Equation (5.1)]{ledford1996statistics}, which was stated with $Y_i$ replaced by $1/[-\log(F(Y_i))]$ (marginally transformed to standard Fr\'echet), once noticing that $-\log \pp{1-\widebar{F}(x)}\sim \widebar{F}(x)$ as $x\rightarrow\infty$.  Hence for $\rho\in (-1,1)$, it follows from \eqref{eq:normal extr indep rate}  that   the bivariate normal $(Y_1,Y_2)$  is extremally independent, with the rate of convergence in \eqref{eq:extr indep} of the order  $\log\pp{1/\widebar{F}(x) }^{-\rho/(1+\rho)}  \widebar{F}(x)^{(1-\rho)/( 1+\rho )}$.

Taking $(X_1,X_2)'$ to be a bivariate Gaussian   with correlation $\rho$ as above, but 
now  consider $Y_i=G(X_i)$ with a  measurable function $G:\R\to \R$, so that the distribution of $Y_i$ is continuous at its upper end point.  What can we say about the extremal dependence of $(Y_1,Y_2)'$? If $G$ is a strictly increasing function,  it is not difficult to see that extremal independence is preserved for $(Y_1,Y_2)'$.
The situation becomes no longer obvious in the absence of monotonicity. In fact, the conclusion remains the same: for an arbitrary transform $G$, we have extremal independence in $(Y_1,Y_2)'$  as long as the correlation $|\rho|<1$ for $(X_1,X_2)'$.  This follows from the following more general fact.
\begin{Cor}\label{Cor:extr indep gaus sub}
Let $H_1$ and $H_2$ be two Gaussian Hilbert spaces with  maximal correlation $\rho:=\rho(H_1,H_2)$  (see \eqref{eq:max corr}).  Suppose $Y_1$ and $Y_2$ are two random variables measurable with respect to $\sigma(H_1)$ and $\sigma(H_2)$ respectively, such that they share the same marginal distribution. Then, for any $x\in \R$,
\begin{equation}\label{eq:Gaus sub joint tail}
\P(Y_1>x, Y_2>x)\le  \P(Y_1>x)^{2/(1+|\rho|)}=\widebar{F}(x)^{2/(1+|\rho|)}.
\end{equation}
In particular, if two random vectors (with all the marginal distributions identical and continuous at the upper end point) are measurable with respect to $\sigma(H_1)$ and $\sigma(H_2)$ respectively, then they are extremally independent if $|\rho|<1$. 
\end{Cor}
\begin{proof}
    This follows directly from Proposition \ref{Pro:hyper holder}.  
\end{proof}
It may be of interest to compare \eqref{eq:normal extr indep rate} with \eqref{eq:Gaus sub joint tail}.  We also mention the following specific consequence in a finite-dimensional setup.

\begin{Cor}\label{Cor:extr indep tran gaus}
Suppose that $X_1=(X_{1,1},\ldots,X_{p,1})'$ and $X_2=(X_{1,2},\ldots,X_{q,2})'$, $p,q\in \Z_+$, are  random vectors such that $(X_1,X_2)$ is jointly Gaussian, and no nontrivial linear combination of components of $X_1$ is   equal to any nontrivial linear combination of components of $X_2$. Suppose   measurable functions $G_1:\R^p\to\R$ and $G_2:\R^q \to \R$   are such that $G_1(X_1)$ shares the same distribution with $G_2(X_2)$  that is continuous at the upper end point.  Then, $G_1(X_1)$ and $G_2(X_2)$ are extremally independent.
\end{Cor}

\begin{proof} 
Let $H_1$ and $H_2$ be the Gaussian Hilbert spaces spanned by the components of $X_1$ and $X_2$ respectively. 
By the assumption, the canonical correlation (see \eqref{eq:can corr})     $r(H_1,H_2)<1$. The conclusion then follows from Proposition \ref{Pro:rho equal} and Corollary \ref{Cor:extr indep gaus sub}.
\end{proof}

As seen above, Proposition \ref{Pro:hyper holder}  is the key to establishing extremal independence. Below, we present yet another quick application of Proposition \ref{Pro:hyper holder}.
Recall the following condition of \cite{leadbetter1983extremes_paper}: for a stationary process $\{Y_k\}$ and a sequence $u_n$ satisfying $\lim_{n \to \infty}n \P\pp{Y_0>u_n}=\tau$ for some $\tau\in (0,\infty)$, we say the condition $D'(u_n)$ holds, if
    \begin{equation}\label{eq:D_n'}
    \lim_{k\rightarrow\infty}    \limsup_{n\rightarrow\infty}  n\sum_{j = 1}^{[n/k]} \proba\pp{Y_0>u_n, Y_j>u_n } =0;
    \end{equation}
     The condition is typically interpreted as an anti-clustering condition, which is famously known to imply the same  weak limit behavior of extremes of $Y_k$ as its i.i.d. counterpart, if combined with the mixing condition $D(u_n)$  also from \cite{leadbetter1983extremes_paper}; see   \cite{leadbetter1983extremes} as well. As shown in \citet[Section 4.4]{leadbetter1983extremes}, the condition $D'(u_n)$ holds if $Y_k$ is Gaussian with covariance $\Cov\pp{Y_k,Y_0}=o\pp{1/\log(k)}$ as $k\rightarrow\infty$. The following result shows that this fact remains true for any instantaneous transform of such a Gaussian sequence. 
\begin{Pro}\label{Pro:D'(u)}
 Suppose    $\{X_k\}$ is a stationary and marginally centered Gaussian process with covariance $\Cov\pp{X_k,X_0}=o\pp{1/\log(k)}$ as $k\rightarrow\infty$. Let $Y_k=G(X_k)$ for some measurable $G:\R\mapsto \R$ and assume $\lim_{n \to \infty} n \P\pp{Y_0>u_n}=\tau$ for some $\tau\in (0,\infty)$. Then the condition   $D'(u_n)$ in \eqref{eq:D_n'} holds for $\{Y_k\}$.
\end{Pro}
The proof of this proposition can be found in Section \ref{sec:pf D'} below.

\subsection{$m$-extremal dependence}\label{sec:m-extremal-dep}

{
Corollary \ref{Cor:extr indep tran gaus} shows that one cannot generate extremal dependence from two
correlated Gaussian vectors unless there is a deterministic linear dependence between them. On the other
hand, in the $(m+1)$-moving-window Gaussian subordination model \eqref{eq:gaus sub Y multi}, extremal
dependence may occur at short lags because of the overlap in the underlying Gaussian inputs. This  
naturally motivates the following notion, which is weaker than the classical $m$-dependence and tailored to extremes.
}

{
Recall that an $\R^d$-valued stationary process $\{Y_k\}_{k\in\Z}$ is called $m$-dependent if
$\sigma(Y_i,\ i\le 0)$ and $\sigma(Y_i,\ i\ge m+1)$ are independent. For instance, this holds in
\eqref{eq:gaus sub Y multi} when the latent process $\{X_i\}$ is i.i.d. In general, however, $m$-dependence
fails as soon as temporal dependence is present in $\{X_i\}$. For our study of extremes, it is enough to require that
observations separated by more than $m$ time units be extremally independent.
}

{
\begin{definition}
Let $\{Y_k=(Y_{k,1},\dots,Y_{k,d})'\}_{k\in\Z}$ be a stationary process in $\R^d$. We say that $\{Y_k\}$
is \emph{$m$-extremally dependent} if for every $k\le 0$, $\ell\ge m+1$, and every $i,j\in\{1,\dots,d\}$,
the pair $(Y_{k,i},Y_{\ell,j})'$ is extremally independent in the sense of \eqref{eq:extr indep}.
\end{definition}
}

We first present a lower bound for the extremal index Definition \ref{Def:ext index multi} under $m$-extremal-dependence. 
\begin{Pro}\label{Pro:m extre dep}
Let $\{Y_k\}$ and $u_n(\tau)$ be as  in Definition \ref{Def:ext index multi}, and suppose that $\{Y_k\}$ is $m$-extremally-dependent and   the limit 
\begin{equation}\label{eq:ext index m-dep}
\theta_m(\tau)=\lim_{n \to \infty}  \P\pp{Y_{1}\le u_n(\tau),\ldots,Y_{m}\le u_n(\tau) \mid Y_{0}\nleq u_{n}(\tau)}
\end{equation}
   exists for all $\tau=(\tau_1,\ldots,\tau_d)'\in (0,\infty)^d$.
 Then for any $\ell> m$, we have $$\theta_\ell (\tau)=\theta_{m}(\tau) \ge \frac{1}{m+1}.$$
\end{Pro}

\begin{proof}
We decompose the event $ \{ Y_{0}\nleq u_{n}(\tau) \}$   by the last exceedance time of $Y_i$ along $i=0,\ldots,2m $ as follows:  
\begin{align}
    &
    \P \left( Y_{0}\nleq u_{n}(\tau)\right)\notag
    \\&=
    \sum_{j=0}^{2m } \P\pp{ Y_{0}\nleq u_{n}(\tau),  
    Y_{j}\nleq u_{n}(\tau),Y_{j+1}\leq u_{n}(\tau),\ldots, Y_{2m}\leq u_{n}(\tau) } \notag 
    \\ &\le
  \sum_{j=0}^{m} \P\left(Y_j\nleq u_n(\tau), Y_{j+1}\leq u_{n}(\tau),\ldots, Y_{2m}\leq u_{n}(\tau)\right)  +\sum_{r=m+1}^{2m}\P\pp{Y_{0}\nleq u_{n}(\tau), Y_{r}\nleq u_{n}(\tau)}\label{eq:D decomp}
  \end{align}
  Note that by stationarity, the $j$th term, $j=0,\ldots,m$, in the first sum in \eqref{eq:D decomp} is equal to
  \begin{align}
      &\P\left(Y_0\nleq u_n(\tau), Y_{1}\leq u_{n}(\tau),\ldots, Y_{2m-j}\leq u_{n}(\tau)\right) 
      = \P\left(Y_0\nleq u_n(\tau), Y_{1}\leq u_{n}(\tau),\ldots, Y_{m}\leq u_{n}(\tau)\right)\\& \quad - \P\left(\bigcup_{l=m+1}^{2m-j} \pc{ Y_0\nleq u_n(\tau),  Y_{1}\leq u_{n}(\tau),\ldots, Y_{m}\leq u_{n}(\tau), Y_{\ell}\nleq u_n(\tau)} \right).\label{eq:Y_0 Y_m}
  \end{align}
   Note that by $m$-extremal-dependence, we have  for $k\ge m+1$ that
   \[
   \P(Y_0\nleq u_n(\tau), Y_{k}\nleq u_n(\tau) )=o\pp{ 1/n},
   \]
   as $n\rightarrow\infty$.  Combining the above relation with \eqref{eq:D decomp}, \eqref{eq:Y_0 Y_m} and the fact that 
   \[
   \P \left(  Y_{0}\nleq u_{n}\right)\ge \P(Y_{0,1}>u_{n,1}(\tau_1))\sim \frac{\tau_1}{n}, \quad \tau_1>0,
   \]
   in view of \eqref{eq:tailbehavemulti_1}, we infer that
\begin{align*}
1 & \le  (m+1) \P\pp{Y_{1}\le u_n(\tau),\ldots,Y_{m}\le u_n(\tau) \mid Y_{0}\nleq u_{n}(\tau)} + o(1),\quad n\rightarrow\infty.
  \end{align*}
 The inequality in the conclusion then follows.  To obtain the equality $\theta_\ell(\tau)=\theta_m(\tau)$, $\ell> m$,  apply an argument   similar to \eqref{eq:Y_0 Y_m} and   $m$-extremal-independence  to conclude
 \begin{align*}
      \P\pp{Y_0\nleq u_n(\tau), Y_1\nleq u_n(\tau),\ldots,Y_\ell\nleq u_n(\tau) } 
     =\P\pp{Y_0\nleq u_n(\tau), Y_1\nleq u_n(\tau),\ldots,Y_m\nleq u_n(\tau) }+o(1/n).
 \end{align*}
\end{proof}

The lower bound $\theta_m(\tau)=1/(m+1)$ is achievable with a moving maxima process  of order $m$ with equal  coefficients; see \citet[Example 10.5]{beirlant2006statistics}. 
 \begin{Rem}
In view of \cite{newell1964asymptotic},  under  $m$-dependence (in the usual sense) in the univariate case, 
 $\theta_m$ in \eqref{eq:ext index m-dep}  matches the extremal index $\theta$  if the limit in \eqref{eq:ext index m-dep} exists. The match $\theta_m=\theta$ is  also known to hold when $\{Y_i\}$ is $m$-extremally-dependent regularly varying process  and satisfies certain mixing-type condition; see, e.g., \citet[Section 7.5]{kulik20heavy}. 
 As will be shown in Theorem \ref{Thm:max} below,  we also have $\theta_m=\theta$     under the Gaussian subordination model \eqref{eq:gaus sub Y multi} with suitable assumptions on the Gaussian process.
\end{Rem}

\begin{Pro}\label{Pro:suff MEV}
A sufficient condition for the limit in \eqref{eq:ext index m-dep} to exist is that the joint distribution of $(Y_0',\ldots
,Y_m')'\in \R^{d(m+1)}$ belongs to the   max-domain of attraction of a    multivariate extreme value distribution.
\end{Pro}
\begin{proof}
The conclusion follows if the limits 
\begin{equation}\label{eq:num limit EVD domain}
\lim_{n \to \infty} n  \P\pp{Y_{1}\le u_n(\tau),\ldots,Y_{m}\le u_n(\tau) , Y_{0}\nleq u_{n}(\tau)}
\end{equation}
and 
\begin{equation}\label{eq:den limit EVD domain}
\lim_{n \to \infty} n  \P\pp{ Y_{0}\nleq u_{n}(\tau)}
\end{equation}
exist and are positive. In view of \citet[Theorem 6.1.5]{deHaanFerreira2006} (see also \citet[Chapter 5]{Resnick1987extreme}),
 this is indeed the case and the limits can be expressed through the exponent measures of  $(Y_0',\ldots
,Y_m')'$ and $Y_0$ respectively,  when $u_{n,i}(\tau_i)=F_i^{\leftarrow}(1-\tau_i/n)$ (for sufficiently large $n$), where $F_i^{\leftarrow}$ is the left-continuous inverse of the CDF $F_i$ of $Y_{0,i}$, $i=1,\ldots,d$. Note that $ 1-F_i(u_{n,i}(\tau_i)) \le  \frac{\tau_i}{n} \le  1-F_i(u_{n,i}(\tau_i)-)  $, and hence $ \lim_{n \to \infty} n \P( Y_{0,i} > u_{n,i}(\tau_i)) = \tau_i$ in view of \eqref{eq:upper end reg}.

We claim that both limits \eqref{eq:num limit EVD domain} and \eqref{eq:den limit EVD domain} then extend to more general  $\wt{u}_n(\tau)=(\wt{u}_{n,1}(\tau_1),\ldots,\wt{u}_{n,d}(\tau_d))'$ that satisfies the relation $\lim_{n \to \infty} n \P( Y_{0,i} > \wt{u}_{n,i}(\tau_i)) = \tau_i$, $i=1,\ldots,d$. Indeed, for the limit \eqref{eq:den limit EVD domain},    by the fact that  $|\P(\cup_{i=1}^d A_i)-\P(\cup_{i=1}^d B_i)|\le \P\pp{ \pp{\cup_{i=1}^d A_i}\Delta \pp{\cup_{i=1}^d B_i}  }\le \sum_{i=1}^d \P\pp{A_i\Delta B_i}$, and the fact that $\P(A\Delta B)=|\P(A)-\P(B)|$ if $A\subset B$ or $B\subset A$, we have
\begin{align*}
n|\P\pp{ Y_{0}\nleq u_{n}(\tau)}-\P\pp{ Y_{0}\nleq \wt{u}_{n}(\tau)}|&\le \sum_{i=1}^d n\P \pp{ \pc{ Y_{0,i} > u_{n,i}(\tau_i)} \Delta  \pc{ Y_{0,i} >  \wt{u}_{n,i}(\tau_i)} }\\
&=\sum_{i=1}^d \left|n\P \pp{ \pc{ Y_{0,i} >  u_{n,i}(\tau_i)}}  -n\P\pp{ \pc{ Y_{0,i} >  \wt{u}_{n,i}(\tau_i)} }\right|\rightarrow 0, 
\end{align*}
as $n\rightarrow\infty$. The limit \eqref{eq:num limit EVD domain} can be handled similarly.
\end{proof}

Now we apply Corollary \ref{Cor:extr indep tran gaus} to a context of  stationary process.

\begin{Pro}\label{Pro:m ext dep gaus sub multi}
  Suppose $\{X_k\}$ is $d$-dimensional centered stationary Gaussian process  with a covariance matrix function $\Gamma(h):=\esp X_h X'_0$ and assume the causal representation \eqref{equality_general_linear_process}, $\sum_r \|\Psi_r\|^2<\infty$, and $\Psi_0$ nonsingular. Let  the subordinated process $\{Y_k\}$ be given as in \eqref{eq:gaus sub Y multi} and suppose that each marginal distribution  of $Y_k$ is continuous at the upper end point.    Then   $\{Y_k\}$   is $m$-extremally-dependent. 
\end{Pro}
\begin{proof}
    By Lemma \ref{le:multiBrochwellDavis} below, the assumption imposed implies that   the $nd\times nd$ covariance matrix $\pp{\Gamma(i-j)}_{1\le i,j\le n}$, that is,   the covariance matrix of the $nd\times 1$ vector $ (X_1',\ldots,X_n')'$,  is nonsingular for all $n\in \Z_+$.  
    This implies that  the $(m+1)d$-dimensional vectors $ (X_{-m}',\ldots,X_{0}')'$ and  $(X_{1}',\ldots,X_{m+1}')'$ are linearly independent.
  Then apply Corollary \ref{Cor:extr indep tran gaus}.
\end{proof}


{
\begin{Lem} \label{le:multiBrochwellDavis}
Suppose that $\{X_k\}_{k\in\Z}$ is a $d$-dimensional centered Gaussian process
with the causal linear representation \eqref{equality_general_linear_process}
with 
$\sum_{r=0}^\infty \|\Psi_r\|^2<\infty$. If $\Psi_0$ is nonsingular, then, for
every $n\in\Z_+$, the $nd\times nd$ covariance matrix
\[
\bigl(\Gamma(i-j)\bigr)_{i,j=1,\dots,n}, \qquad
\Gamma(h)=\esp[X_hX_0'],
\]
is nonsingular.
\end{Lem}
}

{
\begin{proof}
Let $a_1,\dots,a_n\in\mathbb{R}^d$ and suppose that
\[
\Var\left(\sum_{k=1}^n a_k'X_k\right)=0.
\]
It suffices to show that $a_1=\cdots=a_n=0$. Recall
$
X_{k} = \sum_{l=0}^\infty \Psi_{l} \varepsilon_{k-l},,
$
where each  innovation $\varepsilon_k$ has identity covariance matrix. 
  Conditional on \(\sigma(\varepsilon_j:j\le 0)\),
the only random part of \(\sum_{k=1}^n a_k'X_k\) depending on the fresh
innovations \(\varepsilon_1,\dots,\varepsilon_n\) is
\begin{equation*} 
\sum_{m=1}^n 
\left(\sum_{k=m}^n a_k'\Psi_{k-m}\right)\varepsilon_m .
\end{equation*}
Since the total variance is zero, this conditional random part must have
variance zero. Hence we must have 
\begin{equation}\label{eq:cond expre}
\sum_{k=m}^n a_k'\Psi_{k-m}=0, \quad m=1,\ldots,n.
\end{equation}  
 In particular,  with $m=n$ above, we have   \(a_n'\Psi_0=0\). Since \(\Psi_0\) is nonsingular by assumption, this implies  \(a_n=0\).
 
Now, with \(a_n=0\), the expression \eqref{eq:cond expre} when $m=n-1$     becomes
$
a_{n-1}'\Psi_0,
$
and hence we deduce \(a_{n-1}=0\). Continuing backward like this gives
$
a_n=a_{n-1}=\cdots=a_1=0.
$
\end{proof}
}

\subsection{Properties of the Latent Linear Process} \label{se:linear process properties}

Crucial for our result is the second-order behavior of the process $\{ X_k\}$ with causal representation \eqref{equality_general_linear_process}. We  introduce the following lemma that characterizes the decay of the autocovariance matrix function.

\begin{Lem} \label{le:behavior_acf}
Suppose \eqref{equality_long_range_dep_linear_process}. Then, as $h\rightarrow\infty$, 
\begin{equation}\label{eq:b sq tail}
 \sum_{j=0}^\infty  \Psi_{j+h} \Psi_{j+h}' =o\pp{ \frac{1}{\log(h)} }  ,
\hspace{0.2cm}
 \Gamma(h):=\esp X_h X_0'=  \sum_{j=0}^\infty \Psi_j \Psi_{j+h}' =o\pp{ \frac{1}{\log(h)} },
\end{equation}
where asymptotic relations  hold componentwise (and hence also in any matrix norm). 
\end{Lem}

\begin{proof}
Observe that it suffices to show that if instead $\psi_{kl,j}=b_{kl} j^{-1/2}\log(j)^{-1}\ind\{j\ge 4\}$ was assumed for any finite nonnegative constants $b_{kl}$, then both relations in \eqref{eq:b sq tail} hold with $o(\cdot)$ replaced with $O(\cdot)$ in these relations.  Under this  assumption on $\{\Psi_j\}$, 
with an integral comparison, for $k,r=1,\dots,d_0$,    
\begin{align}
\left( \sum_{j=0}^\infty \Psi_j \Psi_{j+h}' \right)_{k,r}
   & \leq
    \sum_{j=0}^\infty \left|\sum_{l=1}^{d_0} \psi_{kl,j} \psi_{lr,j+h} \right| 
    \le  
    \left|\sum_{l=1}^{d_0} b_{kl} b_{lr} \right| \int_e^\infty  x^{-1/2}(x+h)^{-1/2} \log(x)^{-1}\log(x+h)^{-1} dx \nonumber\\
    &= 
    C  \left(\int_{e/h}^e + \int_{e}^\infty \right)  y^{-1/2}(y+1)^{-1/2} \left[\log(hy)\right]^{-1}\left[\log(h(y+1))\right]^{-1 } dy
\nonumber
\\&=: C\pp{ \mathrm{I} _h+ \mathrm{II}_h}. \label{al:behavior_acf_1}
\end{align} 
We consider these two terms separately.
For $\mathrm{I}_h$ in \eqref{al:behavior_acf_1}, we have, by monotonicity, 
\[
\mathrm{I}_h\le  C \int _{e/h}^{e} y^{-1/2}  \log(hy)^{-2 }  dy=O(\log(h)^{-2 }),   
\]
as $h\rightarrow\infty$, where the last relation can be justified by dominated convergence and the Potter bounds for slowly varying functions (e.g., (2.2.1) in \citet{pipiras2017long}): $\log(hy)^{-2}/\log(h)^{-2}\le C y^{-\varepsilon}$ for $y\in (0,e]$, where  $\varepsilon\in (0,1/2)$.
For $\mathrm{II}_h$ in \eqref{al:behavior_acf_1}, by monotonicity,
\begin{align*}
\mathrm{II}_h\le C \int_e^\infty y^{-1} \log(hy)^{-2}dy=\int_{e h}^\infty z^{-1}\log(z)^{-2}dz=O(\log(h)^{-1}).
\end{align*}
Thus, we proved the $O(\cdot)$-version of the second relation in \eqref{eq:b sq tail}. The $O(\cdot)$-version of the first relation in \eqref{eq:b sq tail} follows from an integral comparison similar to the bound for $\mathrm{II}_h$ above.
\end{proof}

\section{Proofs of the Main Results}

\subsection{Proof of Theorem \ref{Thm:max}}\label{sec:pf Thm max}

The proof is based on Theorem \ref{Thm:gap point proc} that involves convergence of a nonstandard point process with gaps between time blocks.
\begin{proof}[Proof of Theorem \ref{Thm:max}]
When applying Theorem \ref{Thm:gap point proc},  we assume throughout that 
$r,p> m$,
where $m$ is as in \eqref{eq:gaus sub Y multi}.  Recall 
   $\pp{ Y_{(j-1)(r+p)+1},\ldots,  Y_{jr+(j-1)p}}$, $j\in \Z_+$, are blocks of consecutive  $Y_k$'s, each of length $r$, with a gap of size $p$  between two adjacent blocks. The state space $E$ is simply $E=\{0,1\}$, where the only bounded subset is $\{1\}$. 
We will use the notion of vague convergence for locally finite measures on the localized Polish space $E=\{0,1\}$; see \citet[Appendix B]{kulik20heavy}.
Choose the sequence $T_n:\R^{d \times r} \to  E$, $n\in \Z_+$, of measurable transforms in \eqref{eq:Y_n(j)} such that 
\begin{equation}\label{eq:Y_n(j)-ind}
Y_n(j)=Y_{r,p,n}(j)= \mathbf{1}{\left\{    Y_{(j-1)(r+p)+1 } \nleq u_n(\tau)  \text{ or }\ldots \text{ or } Y_{jr+(j-1)p} \nleq u_n(\tau) \right\} } \in E, \quad  j\in \Z_+,
\end{equation}
where we have suppressed the dependence on $r,p$ in the notation $Y_{n}(j)$; here and below, we put a block-level time index  (e.g., $j$ in $Y_{n}(j)$)   in  parentheses, in contrast to the ordinary time index in the subscript (e.g., $j$ in $Y_j$). 

We aim to apply Theorem \ref{Thm:gap point proc} and therefore need to verify assumption \eqref{eq:block conv}. Note first that given \eqref{eq:Y_n(j)-ind}, recalling from \eqref{eq:M_n component} that $M_{r,i}=\max_{1\le k\le r}Y_{k,i}$ and $M_r=(M_{r,1},\ldots,M_{r,d})' $,    we have   
\begin{equation} \label{con:max-first-r}
    \P \pp{ Y_{r,p,n}(1) = 1  }=
    \P \pp{ M_r \nleq  u_n(\tau)}.
\end{equation}
By a last-exceedance decomposition, 
\begin{align}
  \P \pp{ M_r \nleq  u_n(\tau)}
=&
\sum_{j=1}^{r}\P \pp{ Y_{j} \nleq u_n(\tau) ,  Y_{j+1}\le u_n(\tau),\ldots, Y_{r}\le u_n(\tau) }  
\nonumber
\\= &
\sum_{j=1}^{r-m}\cdots
+
\sum_{j=r-m+1}^{r}\cdots =: L_n+R_n.
\label{eq:last-exceedance decomposition :L+R}
 \end{align}
We consider the two summands in \eqref{eq:last-exceedance decomposition :L+R} separately. 
For $L_n$ in \eqref{eq:last-exceedance decomposition :L+R}, by an argument similar to  the proof of Proposition \ref{Pro:m extre dep}, we have
 \begin{align}
L_n
&=
\sum_{j=1}^{r-m} \Bigg[ \P \pp{  Y_j\nleq u_n(\tau), Y_{j+1}\leq u_n(\tau),\ldots,Y_{j+m}\leq u_n(\tau)
 }
\\&\hspace{1.2cm}- 
\P \pp{  \bigcup_{\ell=j+m+1}^{r} \pc{ Y_j\nleq u_n(\tau), Y_{j+1}\leq u_n(\tau),\ldots,Y_{j+m}\leq u_n(\tau), Y_\ell \nleq u_n(\tau) }} \Bigg].
\label{eq:L_n decomp}
\\&= (r-m) \P\pp{ Y_0\nleq u_n(\tau), Y_{1}\leq u_n(\tau),\ldots,Y_{1+m}\leq u_n(\tau)} +o(n^{-1})
\end{align}
where for the first term, we have  used stationarity, and for the $o(n^{-1})$ term, we have applied $m$-extremal-dependence via Proposition \ref{Pro:m ext dep gaus sub multi} and \eqref{eq:tailbehavemulti_1}.  
On the other hand, 
by the Assumption \eqref{eq:ext index m-dep multi},  we have 
\[
\lim_{n\rightarrow\infty}\frac{R_n}{\P\pp{Y_0 \nleq u_n(\tau)}} = \sum_{\ell=0}^{m-1} \theta_l(\tau)=:R(\tau)\in (0,m].
\]
Hence combining the relations above, and recalling from \eqref{eq:limit G(tau)} that $\lim_{n \to \infty} n\P\pp{Y_0 \nleq u_n(\tau)}= -\log\pc{H(\tau)}$, $H(\tau) \in (0,1)$, we arrive at the following conclusion regarding \eqref{eq:block conv}:    
\begin{equation}\label{eq:Lambda_rp limit}
\Lambda_{r,p}\pp{\pc{1}}=\lim_{n \to \infty} \frac{n}{r+p} \P \pp{Y_{r,p,n}(1) =1  }= -\pb{\frac{r-m}{r+p} \theta_m(\tau) + \frac{1}{r+p} R(\tau)} \log\pc{H(\tau)}. 
\end{equation}
Next, introduce  
$$J_n(r,p)= \{1,\ldots,n\} \bigcap  \pp{\bigcup_{j=1}^\infty \{(j-1)(r+p)+1,\ldots,  jr+(j-1)p \} }.
$$
In words, $J_n(r,p)$ is the    subset of $\{1,\ldots,n\}$ consisting of   blocks of consecutive integers, each of size $r$, with gap size $p$  between adjacent blocks,  with the last block–gap segment   possibly truncated at  $n$.
Let $J_n(r,p)^c=\{1,\ldots,n\}\setminus J_n(r,p)$.  Write
\begin{align}
E_n(r,p)
&:=
\bigcup_{j\in J_n(r,p)} \pc{ Y_j \nleq u_n(\tau)} 
\subset  
 \pc{M_n\nleq u_n(\tau)}
\\&= \pc{\bigcup_{j\in J_n(r,p)} \pc{ Y_j \nleq u_n(\tau)}} \cup  \pc{\bigcup_{j\in J_n(r,p)^c} \pc{ Y_j \nleq u_n(\tau)}}    
=: E_n(r,p)\cup E_n^c(r,p),
\end{align}
    Therefore, in view of the convergence to Poisson limit in Theorem \ref{Thm:gap point proc}, we have as $ n\rightarrow\infty$ that 
\begin{equation}\label{eq:E_n limit}
\P\pp{E_n(r,p)}\sim \P\pp{\xi_n(\{1\}\times (0,1] )>0}\rightarrow 1-\exp\pb{-\Lambda_{r,p}(\{1\})},
\end{equation} 
Here, to justify the first asymptotic relation above regarding the possible truncation of the last block–gap segment at $n$, note that
\begin{align*}
\{\xi_n(\{1\}\times (0,1] )>0\}  \subset   E_n(r,p) \subset     \{\xi_n(\{1\}\times (0,1] )>0\} \cup D_{n}(r,p),
\end{align*}
where the event $D_n(r,p):=\bigcup_{  (r+p)\lfloor \frac{n}{r+p} \rfloor   +1 \le   \ell \le  (r+p)\lfloor \frac{n}{r+p} \rfloor+r } \{ Y_\ell \nleq u_n(\tau) \}$ has probability tending to $0$ in view of a union bound, stationarity and \eqref{eq:tailbehavemulti_1}. 

Next, observe that the event $E_n^c(r,p)$,  as $n\rightarrow\infty$, behaves like those of   $E_n(p,r)$, i.e.,  the roles between ``blocks'' and ``gaps'' are switched.  Using arguments similar to those above, we also have
\begin{equation}\label{eq:E_n^c limit}
\P\pp{E_n^c(r,p)} \rightarrow 
1-\exp(-\Lambda_{p,r}(\{1\})),
 \end{equation}
 as $n\rightarrow\infty$, noticing the switch between $r$ and $p$ compared to  $\Lambda_{r,p}$ in \eqref{eq:Lambda_rp limit}.
 
 Finally, from \eqref{eq:Lambda_rp limit}, we see that $\lim_{r\rightarrow\infty}\Lambda_{r,p}(\{1\})=-\log\pc{ H(\tau)^{\theta(\tau)}}$, and $\lim_{r\rightarrow\infty}\Lambda_{p,r}(\{1\})=0$. 
 So the conclusion \eqref{eq:sample_extreme_max_res} of the theorem follows  from \eqref{eq:E_n limit} and \eqref{eq:E_n^c limit} by taking limits on both ends of the following inequalities: 
 \begin{align*}
 \P\pp{E_n(r,p)} \le \P\pp{ M_n\nleq u_n(\tau)} \le  \P\pp{E_n(r,p)} +\P\pp{E_n^c(r,p)}, 
 \end{align*}
first, as $n\rightarrow\infty$, and then $r\rightarrow\infty$. 
\end{proof}

\subsection{Proof of Theorem \ref{Thm:limit max MA}}

In this section, we illustrate the use of Theorem \ref{Thm:max} through proving Theorem \ref{Thm:limit max MA}. Since Theorem \ref{Thm:max} is not readily available because the M4 process in \eqref{eq:gaus sub max MA} is a transformation in infinitely many variables, we need to apply a truncation argument to transfer a limit theorem for finite-dimensional transforms as in Theorem \ref{Thm:max} to that for   infinite-dimensional  ones.
In particular,  the following approximation strategy adapted from \citet[Proposition 1.4]{chernick1991calculating} (see also \citet[Corollary 7.5.7]{kulik20heavy})  will serve this purpose.

\begin{Lem}\label{Lem:approx}
Suppose for each $m\in \Z_+$, the process $\{(Z_k,Z_k^{(m)})\}_{k\in \Z}$ is  jointly stationary with $Z_k=(Z_{k,1},\ldots,Z_{k,d})'\in \R^d$ and $Z_{k}^{(m)}=(Z_{k,1}^{(m)},\ldots,Z_{k,d}^{(m)})'$, $ d\in\Z_+$.  Suppose the sequence  $u_n =(u_{n,1},\dots,u_{n,d})'$ satisfies  
 for any $\epsilon>0$ that
\begin{equation}\label{eq:approx Lem cond 1}
\lim_{\epsilon\rightarrow 0}
\limsup_{n\rightarrow\infty }n 
\P\pp{(1-\epsilon)u_{n,i}< Z_{0,i} \le (1+\epsilon)u_{n,i}}=0, 
\end{equation}
and  
\begin{equation}\label{eq:approx Lem cond 2}
\lim_{m\rightarrow \infty}\limsup_{n\rightarrow\infty }n   \P\pp{ | Z_{0,i} - Z^{(m)}_{0,i} | > \epsilon u_{n,i} } =0, 
\end{equation}
for all $i=1,\ldots,d$.
Write $M_i(J)=\max_{j\in J} Z_{j,i}$ and $M^{(m)}_i(J)=\max_{j\in J} Z_{j,i}^{(m)}$ for nonempty $J\subset \{1,\ldots,n\}$,   $M(J)=\pp{M_1(J), \ldots, M_d(J)}'$ and $M^{(m)}(J)=\pp{M^{(m)}_1(J),\ldots,M^{(m)}_d(J) }'$.
Then,
\begin{equation}\label{eq:tri approx max}
\lim_{m\rightarrow\infty}\limsup_{n\rightarrow\infty}   \sup_{\emptyset\neq J \subset\{1,\ldots,n\}}  \left|   \P\pp{M(J)\le u_n} -\P\pp{M^{(m)}(J)\le u_n   }\right| =0.
\end{equation}
If additionally, $0<\liminf_{n\rightarrow\infty} n \P(Z_{0}\nleq u_{n})\le \limsup_{n\rightarrow\infty} n \P(Z_{0}\nleq u_{n})<\infty $, then
for any $\ell\in \Z_+$, 
\begin{equation} \label{eq:tri approx theta}
\begin{aligned}
\lim_{m\rightarrow\infty}\limsup_{n\rightarrow\infty}  
&
\Bigg| \P  \pp{ 
Z_1\le u_n,\ldots,  Z_\ell \le u_n  ~\bigg|  ~ Z_0\nleq u_n } 
\\&\vspace{1cm}-  \P  \pp{ 
Z_1^{(m)}\le u_n,\ldots,  Z_\ell^{(m)} \le u_n  ~\bigg| ~  Z_0^{(m)}\nleq u_n }  \Bigg|=0.
\end{aligned}
\end{equation}
\end{Lem}

\begin{proof}
We prove \eqref{eq:tri approx max} and \eqref{eq:tri approx theta} separately.

\medskip
\noindent\textit{Proof of \eqref{eq:tri approx max}:}

The proof is a direct adaptation of \citet[Proposition 1.4]{chernick1991calculating}.  For the reader's convenience we include the details here. Below, the absolute value sign on a vector is understood component-wise.
    \begin{align}
    &
    \left|   \P\pp{M(J)\le u_n} -\P\pp{M^{(m)}(J)\le u_n   }\right|
    \nonumber
    \\&\leq
    \P\pp{  M(J)\le u_n,  M^{(m)}(J)\nleq u_n,    | M(J) - M^{(m)}(J) |   \le  \epsilon u_{n}}  
    \nonumber\\&\hspace{1cm}+  
    \P\pp{  M(J)\nleq u_n,  M^{(m)}(J)\le  u_n, | M(J) - M^{(m)}(J) |  \le  \epsilon u_{n}}   + \P\pp{   | M(J) - M^{(m)}(J) |  \nleq \epsilon u_{n} }
        \\&\leq     \sum_{i=1}^d \P\pp{
        (1-\epsilon)u_{n,i}  < M_{i}(J) \le (1+\epsilon)u_{n,i} } + \sum_{i=1}^d    \P\pp{ | M_{i}(J) - M_{i}^{(m)}(J) | > \epsilon u_{n,i} } \nonumber
        \\&\le \sum_{i=1}^d  n \P\pp{
        (1-\epsilon)u_{n,i}  < Z_{0,i} \le (1+\epsilon)u_{n,i} } + \sum_{i=1}^d   n \P\pp{ | Z_{0,i} - Z_{0,i}^{(m)} | > \epsilon u_{n,i} } \label{eq:approx lem inter}
    \end{align}
The desired result then follows by \eqref{eq:approx Lem cond 1} and \eqref{eq:approx Lem cond 2}.

\medskip

\noindent
\textit{Proof of \eqref{eq:tri approx theta}:}
Introduce the following events
\[
A_n := \left\{Z_1\le u_n,\ldots,  Z_\ell \le u_n  \right\}, \quad B_n := \left\{  Z_0\nleq u_n\right\},
\]
\[
A_n^{(m)} := \left\{Z_1^{(m)}\le u_n,\ldots,  Z_\ell^{(m)} \le u_n  \right\}, \quad B_n^{(m)} := \left\{  Z_0^{(m)}\nleq u_n\right\}.
\]

Then, we aim to show that
\begin{equation} \label{eq:rewritten:approx-eq1}
    \lim_{m \to \infty} \limsup_{n \to \infty} n \left| \P(A_n \mid B_n) - \P(A_n^{(m)} \mid B_n^{(m)}) \right| = 0. 
\end{equation}
The difference in \eqref{eq:rewritten:approx-eq1} can be written as
\begin{equation} \label{eq:rewritten:approx-eq2}
\left| \frac{\P(A_n \cap B_n)}{\P(B_n)} - \frac{\P(A_n^{(m)} \cap B_n^{(m)})}{\P(B_n^{(m)})} \right| \le R_1 + R_2,
\end{equation}
where
\[
R_1 := \left| \frac{\P(A_n \cap B_n) - \P(A_n^{(m)} \cap B_n^{(m)})}{\P(B_n)} \right|, \quad
R_2 := \left| \P(A_n^{(m)} \cap B_n^{(m)}) \left( \frac{1}{\P(B_n)} - \frac{1}{\P(B_n^{(m)})} \right) \right|.
\]
We need to show $R_1$ and $R_2$ in \eqref{eq:rewritten:approx-eq2} vanish as $n\rightarrow\infty$.

First,  we have
\begin{equation}\label{eq:An Anm}
\lim_{m\rightarrow\infty}\limsup_{n\rightarrow\infty}  n \left|   \P\pp{ A_n} -\P\pp{ A_n^{(m)} }\right| =0,
\end{equation}
which can be derived similarly as the proof of \eqref{eq:tri approx max} with $J$ fixed to be $\{1,\ldots,\ell\}$, and $n$ in the   union bound \eqref{eq:approx lem inter} replaced by $\ell$.  In the same way,  we also infer that
\begin{equation}\label{eq:An B Anm Bm}
\lim_{m\rightarrow\infty}\limsup_{n\rightarrow\infty}  n \left|   \P\pp{ A_n\setminus B_n } -\P\pp{ A_n^{(m)}\setminus B_n^{(m)} }\right| =0.
\end{equation}
Hence combining \eqref{eq:An Anm} and \eqref{eq:An B Anm Bm}, we have
\begin{equation}\label{eq:A_n cap B_n and m ver}
\lim_{m\rightarrow\infty}\limsup_{n\rightarrow\infty}  n |\P(A_n \cap B_n) - \P(A_n^{(m)} \cap B_n^{(m)})|=0. 
\end{equation}
This combined with the assumption that $\liminf_{n\rightarrow\infty} n \P(B_n)>0$ implies $\lim_{n \to \infty} R_1=0$.

Now for $R_2$,  similar to \eqref{eq:An Anm}, we infer that 
\begin{align}\label{eq:Bn Bnm}
   \lim_{m\rightarrow\infty}\limsup_{n\rightarrow\infty}  n \left|   \P\pp{ B_n} -\P\pp{ B_n^{(m)} }\right| =0, 
\end{align}
which combined with the assumption $\liminf_{n\rightarrow\infty}  n\P\pp{B_n}>0$ implies that 
for all $m$ sufficiently large,
\begin{equation}\label{eq:B_n^m lower bound}
\liminf_{n\rightarrow\infty} n  \P\pp{ B_n^{(m)} } > \epsilon 
\end{equation}
for some $\epsilon>0$. Also from \eqref{eq:A_n cap B_n and m ver} and the assumption $\limsup_{n\rightarrow\infty} n\P\pp{B_n}<\infty$, we infer   for all $m$ sufficiently large that
\begin{equation}\label{eq:A_nm B_nm bound}
    \limsup_{n\rightarrow\infty} n  \P\pp{ A_n^{(m)} \cap B_n^{(m)} } <M
\end{equation}
for some $M<\infty$.
Therefore,  combining the assumption  $\liminf_{n\rightarrow\infty} n \P(B_n)>0$, \eqref{eq:Bn Bnm}, \eqref{eq:B_n^m lower bound} and \eqref{eq:A_nm B_nm bound}, we have
\begin{align*}
\lim_{m\rightarrow\infty}\limsup_{n\rightarrow\infty} R_2=   \lim_{m\rightarrow\infty}\limsup_{n\rightarrow\infty}   n\P\pp{A_n^{(m)}\cap B_n^{(m)}}   \frac{ n\left|\P\pp{B_n}- \P\pp{B_n^{(m)}} \right| }{n\P\left(B_n\right) n\P\pp{B_n^{(m)}}}=0.
\end{align*}
\end{proof}

To establish Theorem \ref{Thm:limit max MA}, we apply Lemma \ref{Lem:approx}. 
Introduce for each $k \in \Z$, $i=1,\dots,d$, $m\in \Z_+$, a truncated version of ${Y^{\mathrm{M4}}_{k,i}}$ and its remainder as
\begin{equation} \label{eq:Y-m,Ym}
{
   Y_{k,i}^{\MF,(m)}= \bigvee_{|r|\le  m} \bigvee_{j=1}^d a_{ij,r} W_{k-r,j},\quad 
   Y_{k,i}^{\MF,(-m)}= \bigvee_{|r|> m} \bigvee_{j=1}^d a_{ij,r} W_{k-r,j}, 
   }
\end{equation}
{
Write $Y_k^{\MF,(m)}=(Y_{k,1}^{\MF,(m)},\ldots,Y_{k,d}^{\MF,(m)})'$ and $Y_k^{\MF,(-m)}=(Y_{k,1}^{\MF,(-m)},\ldots,Y_{k,d}^{\MF,(-m)})'$.
}

{The following lemma computes the exponent function and the truncated extremal index for the M4 model.}

\begin{Lem}\label{Lem:MA tail RV}
Recall M4 process $\{ {Y^{\mathrm{M4}}_{k,i}}\}$  from \eqref{eq:gaus sub max MA}, where the coefficients satisfy the summability condition \eqref{eq:cond1-maxsum}.   
Then the following conclusions hold.
\begin{enumerate}[label=(\alph*)]
\item \label{item:Y_k finite}  ${Y^{\mathrm{M4}}_{k,i}}<\infty$ almost surely, $k\in \Z$, $i=1,\ldots,d$. 
\item \label{item:Y_k tail} Let $u_n(\tau)=(u_{n,1}(\tau_1),\ldots,u_{n,d}(\tau_d))'$ be as in Definition \ref{Def:ext index multi}  {with $\{Y_k\}=\pc{Y_k^{\MF}}$}. Then 
\begin{align}
    \lim_{n\rightarrow\infty } n
    \P\pp{ {Y^{\mathrm{M4}}_k} \nleq  u_{n}(\tau) } 
    = 
  \sum_{r \in \Z} \sum_{j =1}^d \bigvee_{i = 1}^d  \frac{a_{ij,r}^\alpha \tau_i}{A_i} ,
    \label{eq:approx Lem cond 1-proof-eq3}
\end{align}
where $A_i$ is as in \eqref{eq:A_i}.
\item \label{item:Y_km tail} For the same $u_n(\tau)=(u_{n,1}(\tau_1),\ldots,u_{n,d}(\tau_d))'$ as in \ref{item:Y_k tail},
\begin{align}
    \lim_{n\rightarrow\infty } n
    \P\pp{ {Y_k^{\MF,(m)}} \nleq  u_{n}(\tau) } 
    = 
  \sum_{|r| \le m } \sum_{j =1}^d \bigvee_{i = 1}^d  \frac{a_{ij,r}^\alpha \tau_i}{ {A_i}} ,
    \label{eq:approx Lem cond 1-proof-eq3 m}
\end{align}
\item \label{item:Y_km ext ind}   For the same $u_n(\tau)=(u_{n,1}(\tau_1),\ldots,u_{n,d}(\tau_d))'$ as in \ref{item:Y_k tail},
\begin{align}
    \theta_{2m}(\tau)&:=
    \lim_{n\rightarrow \infty} 
    \P  \pp{  Y_{1}^{\MF,(m)}\le u_n(\tau),\ldots, Y_{2m}^{\MF,(m)}\le u_n(\tau) ~ \bigg|  ~
     Y_{0}^{\MF,(m)}\nleq u_n(\tau) } \nonumber  
   \\&   =
    \frac{
    \sum_{j = 1}^d   {\bigvee_{|r| \leq m} \bigvee_{i = 1}^d  \left(a_{ij,r}^\alpha \tau_i \right)/ {A_i} }} {{
    \sum_{|r| \leq m}} \sum_{j = 1}^d  \bigvee_{i = 1}^d  \left(a_{ij,r}^\alpha \tau_i\right)/ {A_i} }.\label{eq:theta_2m}
    \end{align}
\end{enumerate}

\begin{Rem}
Although the limit in \eqref{eq:approx Lem cond 1-proof-eq3 m} has the same form as that in \eqref{eq:approx Lem cond 1-proof-eq3}, except that the summation is restricted to $|r|\le m$, the conclusion in (c) does not follow directly from that in (b), since the same threshold $u_n(\tau)$ associated with the untruncated process $Y_k^{\MF}$ appears in both (b) and (c). Both conclusions (c) and (d) for the truncated process $Y_k^{\MF,(m)}$ are required in the proof of Theorem \ref{Thm:limit max MA}. 
 \end{Rem}

\begin{Rem}
We have $\tau_1\vee\ldots\vee \tau_d\le \sum_{r \in \Z} \sum_{j =1}^d \bigvee_{i = 1}^d  \frac{a_{ij,r}^\alpha \tau_i}{A_i}\le \tau_1+\ldots+\tau_d$, 
where the first inequality follows from exchange of the sums $ \sum_{r \in \Z} \sum_{j =1}^d$ with the maximum  $\bigvee_{i = 1}^d$, and the second follows from bounding the maximum $\bigvee_{i = 1}^d $ by a sum $\sum_{i = 1}^d$. This is consistent with \eqref{eq:G bound}.
\end{Rem}
\end{Lem}

\begin{proof}[Proof of Lemma \ref{Lem:MA tail RV}]
Note that $W_{k,j}$'s in \eqref{eq:gaus sub max MA} are marginally identically distributed  by construction. We introduce $W\overset{d}{=} W_{k,j}$.

\medskip
\noindent \textit{Proof of \ref{item:Y_k finite}.}

   For $y\in (0,\infty)$,  
due to the assumption \eqref{eq:alphavarying}, and Potter bounds (e.g., \citet[Proposition 1.4.1]{kulik20heavy}), we have  
\begin{equation}\label{eq:single term tail up bd}
\P\pp{ a_{ij,r} W >y}\le C a_{ij,r}^{\alpha-\epsilon} y^{-\alpha+\epsilon},
\end{equation}
where $C>0$ is a constant that does not depend on $i$, $j$, $r$ or $y$.  So we have by a union bound that 
\begin{align}\label{eq:Y_k,i tail up bd}
\P\pp{Y_{k,i}^{\MF}>y}\le \sum_{r\in \Z}\sum_{j=1}^d \P\pp{ a_{ij,r} W >y}\le  C \pp{\sum_{r\in \Z}\sum_{j=1}^d a_{ij,r}^{\alpha-\epsilon}} y^{-\alpha+\epsilon},
\end{align}
which tends to $0$ as $y\uparrow\infty$ since the expression in the parenthesis in the last bound  is finite due to \eqref{eq:cond1-maxsum}. So it follows that  $Y_{k,i}<\infty$ almost surely.

\medskip

 \noindent \textit{Proof of \ref{item:Y_k tail}.}

First, we claim that
\begin{align}\label{eq:Y_k,i marg tail}
  \P\pp{Y^{\mathrm{M4}}_{k,i}>y}\sim  \sum_{r\in \Z}\sum_{j=1}^d \P\pp{ a_{ij,r} W >y}\sim A_i y^{-\alpha}  L(y) .
\end{align}
 Indeed, 
define for each $k \in \Z$, $i=1,\dots,d$, $m\in \Z_+$, a truncated version of $ Y_{k,i}$  as
\begin{equation} \label{eq:Ym}
   Y_{k,i}^{\MF,(m)}= \bigvee_{|r|\le  m} \bigvee_{j=1}^d a_{ij,r} W_{k-r,j}.
\end{equation}
  Then,  
\begin{equation}\label{eq:Y_k,i tail lw bd}
\P\pp{Y^{\mathrm{M4}}_{k,i}>y}\ge \P\pp{ Y_{k,i}^{\MF,(m)}>y} \sim  \sum_{|r|\le m }\sum_{j=1}^d \P\pp{ a_{ij,r} W >y},  
\end{equation}
{as $y\to\infty$, where the last relation follows from the extremal independence of the family
$\{W_{k-r,j}\}_{|r|\le m,\;j=1,\dots,d}$, in the sense that the probability of two distinct terms
$a_{ij,r}W_{k-r,j}$ and $a_{ij,r'}W_{k-r',j'}$ simultaneously exceeding $y$ is negligible.
}
Then,  the first relation in \eqref{eq:Y_k,i marg tail} follows from the first relation in \eqref{eq:Y_k,i tail up bd} and \eqref{eq:Y_k,i tail lw bd} by first letting $y\rightarrow\infty$ and then $m\rightarrow\infty$,  and the second relation in \eqref{eq:Y_k,i marg tail} can be justified by dominated convergence  theorem and Potter bounds; see, e.g., the proof of \citet[Theorem 2.3]{hsing1986extreme}.

Now, combining \eqref{eq:Y_k,i marg tail} with the fact that $\lim_{n \to \infty} n\P(Y^{\mathrm{M4}}_{k,i}>u_{n,i}(\tau_i))=\tau_i\in (0,\infty)$, in view of \citet[Theorem 1.5.12]{bingham1989regular} (see also \citet[Section 1.3]{kulik20heavy}) 
we know that 
\begin{equation}\label{eq:u_{n,i} asymp}
u_{n,i}(\tau_i)\sim  \pp{\frac{A_i n}{\tau_i}}^{1/\alpha}    \ell(n),
\end{equation}
as $n\rightarrow\infty$, for some slowly varying  $\ell(n)$.
Set 
\begin{equation}\label{eq:b_{r,j}}
b_{r,j,n} (\tau)  = \bigvee_{i = 1}^d  a_{ij,r}/u_{n,i}(\tau_i),
\end{equation}  
where we can and shall assume $u_{n,i}(\tau_i)>0$, for all $i=1,\ldots,d$ and $n\in \Z_+$. We then have
\begin{equation}\label{eq:b u ratio}
 \lim_{n\rightarrow\infty }b_{r,j,n} (\tau) u_{n,j}(1)   = A_j^{1/\alpha} \bigvee_{i=1}^d \frac{a_{ij,r} \tau_i^{1/\alpha}}{A_i^{1/\alpha}}.
\end{equation}
Furthermore, from \eqref{eq:Y_k,i marg tail}, we also infer that
\begin{equation}\label{eq:W u_n tail}
\lim_{n \to \infty} n\P\pp{ W > u_{n,j}(\tau_j) } =\frac{\tau_j}{A_j},\quad j=1,\ldots,d.
\end{equation}

Now we proceed to show \eqref{eq:approx Lem cond 1-proof-eq3}. 
We have,  as $n\rightarrow\infty$,
\begin{align*}
\P\left( Y_0^{\MF} \nleq  u_n(\tau)\right) 
&= 
\P\left( \bigcup_{i=1}^d  \pc{ \bigvee_{r\in \Z} \bigvee_{j=1}^d    a_{ij,r} W_{-r,j} >  u_{n,i}(\tau) }\right)
\sim 
\sum_{r\in \Z} \sum_{j = 1}^d \P\left( b_{r,j,n} (\tau) W >  1 \right),
\end{align*}
where  the  last   step  follows from an argument similar to  that for \eqref{eq:Y_k,i marg tail}. 
Then,
\begin{align}
   \lim_{n \to \infty} n\P\left( Y_0^{\MF} \nleq   u_n(\tau) \right)
    &=
    \lim_{n \to \infty} n\sum_{r \in \Z } \sum_{j = 1}^d \P\left( b_{r,j,n} (\tau) W  > 1\right)
    \\&= \lim_{n \to \infty}
    \sum_{r \in \Z } \sum_{j = 1}^d 
    \frac{\P\left( b_{r,j,n} (\tau) W  > 1 \right)}{
    \P\left( W/u_{n,j}(1) > 1 \right)
    }
    n\P\left( W /u_{n,j}(1) > 1 \right)
    \\&= 
    \sum_{r \in \Z } \sum_{j = 1}^d   \left( \bigvee_{i = 1}^d  \frac{a_{ij,r}^\alpha \tau_i}{A_i}\right),
\end{align}
where in the last step, we first interchange the limit and the summation, which can be justified by the dominated convergence theorem  and Potter bounds in a similar fashion as in the first part of the proof, and then
apply  the regular variation of $W$, \eqref{eq:b u ratio} and \eqref{eq:W u_n tail}. 

\medskip
 \noindent \textit{Proof of \ref{item:Y_km tail}.} 

It follows from a  calculation similar to that for $ \lim_{n \to \infty} n\P\left( Y_0^{\MF} \nleq   u_n(\tau) \right)$ above with an obvious modification.

\medskip
\noindent \textit{Proof of \ref{item:Y_km ext ind}.}

With further explanations given below, we have
\begin{align}
    &
    \lim_{n\rightarrow \infty}  
    \P  \pp{  Y_{1}^{\MF,(m)}\le u_n(\tau),\ldots, Y_{2m}^{\MF,(m)}\le u_n(\tau)  \bigg|  
     Y_{0}^{\MF,(m)}\nleq u_n(\tau) }
    \nonumber
\\&=  
    \lim_{n\rightarrow \infty} \frac{n\P  \pp{  Y_{1}^{\MF,(m)}\le u_n(\tau),\ldots, Y_{2m}^{\MF,(m)}\le u_n(\tau),
     Y_{0}^{\MF,(m)}\nleq u_n(\tau) }}{ n\P \pp{  Y_{0}^{\MF,(m)}\nleq u_n(\tau) } }
    \label{al:eq1-denum-num}
\\&=
      \frac{
    \sum_{j = 1}^d  \bigvee_{|r| \leq m} \bigvee_{i = 1}^d  \left(a_{ij,r}^\alpha \tau_i \right)/A_i }{
    \sum_{|r| \leq m} \sum_{j = 1}^d  \bigvee_{i = 1}^d  \left(a_{ij,r}^\alpha \tau_i\right)/A_i }.   \label{al:eq2-denum-num}
\end{align}
The denominator in \eqref{al:eq1-denum-num} converges to the denominator in \eqref{al:eq2-denum-num} due to \ref{item:Y_km tail}. We are left to show that the numerator in \eqref{al:eq1-denum-num} converges to the numerator in \eqref{al:eq2-denum-num}.
  
Recall $b_{r,j,n}(\tau)= \bigvee_{i = 1}^d  a_{ij,r}/u_{n,i}(\tau_i)$ from \eqref{eq:b_{r,j}} and set
\begin{equation*}
b_{r,j,n}^{(m)}(\tau) := b_{r,j,n} (\tau)  \mathbf{1}{\{ | r | \leq m \}} 
\text{ and }
B_{r,j,n}^{(m)}(\tau) := \max_{1\le k\le 2m} b_{k+r,j,n}^{(m)}(\tau),\quad r\in \Z, \ j=1,\ldots,d, \ m,n \in \Z_+.
\hspace{0.2cm}
\end{equation*}
  Then,   the event  inside the probability sign in the numerator of \eqref{al:eq1-denum-num} can be expressed as
\begin{align}
    &
    \pc{  Y_{1}^{\MF,(m)}\le u_n(\tau),\ldots, Y_{2m}^{\MF,(m)}\le u_n(\tau),
     Y_{0}^{\MF,(m)}\nleq u_n(\tau) }
    \label{eq:event1}
    \\&=
    \left\{  \bigvee_{r=1}^{2m} \bigvee_{l =1}^{d} \frac{Y^{\MF,(m)}_{r,l}}{u_{n,l}(\tau_l)}\le 1  ,  \bigvee_{|s| \le m} \bigvee_{i=1}^d \bigvee_{t=1}^d \frac{a_{it,s} W_{ -s,t}}{u_{n,i}(\tau_i)} > 1 \right\}
    \nonumber
    \\&=
    \bigcup_{|s| \le m} \bigcup_{t=1}^d
    \left\{ \mathcal{A}_{s,t}(n) \cap \mathcal{B}_{s,t}(n) \cap \mathcal{C}_{s,t}(n) \cap \mathcal{D}_{s,t}(n) \right\},
\end{align}
where
\begin{align*}
 \cl{A}_{s,t}(n)&=\left\{ \frac{1}{b^{(m)}_{s,t,n}(\tau)} < W_{-s,t} \le \frac{1}{B^{(m)}_{s,t,n}(\tau)}\right\},\\
 \cl{B}_{s,t}(n)&=\left\{
     W_{-l,j} \le \frac{1}{B^{(m)}_{l,j,n}(\tau)}, l \neq s, j \neq t, |l|\le m, j=1,\ldots,d
      \right\},\\
      \cl{C}_{s,t}(n)&=   \left\{
      W_{-s,j} \le \frac{1}{B^{(m)}_{s,j,n}(\tau)},  j\neq t, j=1,\ldots,d
      \right\}\\
      \cl{D}_{s,t}(n)&= \left\{
      W_{-l,t} \le \frac{1}{B^{(m)}_{l,t,n}(\tau)}, l \neq s, |l|\le m
      \right\},  
\end{align*}
where if a denominator is $0$ in a ratio where the numerator is $1$, the ratio is understood as $\infty$.

We claim that, as $n\rightarrow\infty$,
\begin{align} 
    &
    n\P\pp{  Y_{1}^{\MF,(m)}\le u_n(\tau),\ldots, Y_{2m}^{\MF,(m)}\le u_n(\tau),
     Y_{0}^{\MF,(m)}\nleq u_n(\tau) }
    \nonumber
    \\&=
    n\P\left( \bigcup_{|s| \le m} \bigcup_{t=1}^d
    \left\{ \mathcal{A}_{s,t}(n) \cap \mathcal{B}_{s,t}(n) \cap \mathcal{C}_{s,t}(n) \cap \mathcal{D}_{s,t}(n) \right\} \right)
    \nonumber
    \\&\sim
    n \sum_{|s| \leq m} \sum_{t=1}^d \P\left( 
  \mathcal{A}_{s,t}(n) \cap \mathcal{B}_{s,t}(n) \cap \mathcal{C}_{s,t}(n) \cap \mathcal{D}_{s,t}(n)   \right)
    \label{al:joint-prob-1}
    \\&\sim
    n \sum_{|s| \leq m} \sum_{t=1}^d \P\left( 
    \mathcal{A}_{s,t}(n) \right)
    \label{al:joint-prob-2}
    \\&\sim
    \sum_{t=1}^d 
    \bigvee_{|r| \leq m} \bigvee_{i = 1}^d (a_{it,r}^\alpha \tau_i)/A_i,
    \label{al:joint-prob-3}
\end{align}
where we prove the asymptotics \eqref{al:joint-prob-1}--\eqref{al:joint-prob-3} below.

\noindent\textit{Proof of \eqref{al:joint-prob-2}.}

We aim to approximate the probability in \eqref{al:joint-prob-1} by controlling the lower and upper bounds. An upper bound is simply given by
\begin{align}
    \P\left(
 \mathcal{A}_{s,t}(n) \cap \mathcal{B}_{s,t}(n) \cap \mathcal{C}_{s,t}(n) \cap \mathcal{D}_{s,t}(n)  \right) 
    \le 
    \P\left( \mathcal{A}_{s,t}(n) \right).
\end{align}
For a lower bound note that
\begin{align}
    &
    \P\left( 
    \left\{ \mathcal{A}_{s,t}(n) \cap \mathcal{B}_{s,t}(n) \cap \mathcal{C}_{s,t}(n) \cap \mathcal{D}_{s,t}(n) \right\} \right)
    \nonumber
    \\&=
    \P\left(
    \left\{ \mathcal{A}_{s,t}(n) \backslash (
    \mathcal{A}_{s,t}(n) \backslash \mathcal{B}_{s,t}(n) \cup \mathcal{A}_{s,t}(n) \backslash \mathcal{C}_{s,t}(n) \cup 
    \mathcal{A}_{s,t}(n) \backslash \mathcal{D}_{s,t}(n)) \right\} \right)
    \nonumber
    \\ &\geq
    \P\left( \mathcal{A}_{s,t}(n) \right) 
    -
    \P\left( \mathcal{A}_{s,t}(n) \backslash \mathcal{B}_{s,t}(n) \right) 
    -
    \P\left( \mathcal{A}_{s,t}(n) \backslash \mathcal{C}_{s,t}(n) \right) 
    -
    \P\left( \mathcal{A}_{s,t}(n) \backslash \mathcal{D}_{s,t}(n) \right)
    \label{al:joint-prob-4}
    \\ &=
    \P\left( \mathcal{A}_{s,t}(n) \right) 
    +
    o(n^{-1}),
    \label{al:joint-prob-5}
\end{align}
as $n\rightarrow\infty$, where \eqref{al:joint-prob-5} can be inferred by bounding the probabilities in \eqref{al:joint-prob-4} separately. Indeed, we have by a union bound that
\begin{align}
    \P\left( \mathcal{A}_{s,t}(n) \backslash \mathcal{B}_{s,t}(n) \right)
     \leq
     \sum_{l \neq s, |l|\le m, j\neq t, j=1,\ldots,d}
    \P\left( 
    \left\{  W_{-s,t}> \frac{1}{b^{(m)}_{s,t,n}(\tau)} ,
     W_{-l,j}>\frac{1}{B^{(m)}_{l,j,n}(\tau)} 
    \right\}
    \right).
    \label{al:joint-prob-6}
\end{align}
Then each probability term above is $o(n^{-1})$ due to the extremal independence between  $  W_{-s,t}$ and $  W_{-l,j}$,  and the fact that $\P\left( 
    {b_{s,t,n}(\tau) W> 1} \right)=O(n^{-1}) $ shown in the proof of \ref{item:Y_k tail}.
The remaining probabilities  $\P\left( \mathcal{A}_{s,t}(n) \backslash \mathcal{C}_{s,t}(n) \right)$ and $\P\left( \mathcal{A}_{s,t}(n) \backslash \mathcal{D}_{s,t}(n) \right)$ in \eqref{al:joint-prob-4}  can be handled similarly.

\noindent
\textit{Proof of \eqref{al:joint-prob-3}.}

Introduce $a_{ij,r}^{(m)}=a_{ij,r}  \mathbf{1}{\{ | r | \leq m \}} $.
Then
\begin{align}
n\P\left( \mathcal{A}_{s,t}(n) \right)  
&=
\pb{n\mathbb{P}\left( b^{(m)}_{s,t,n}(\tau) W_{-s,t} >1  \right) - 
n\mathbb{P}\left( B^{(m)}_{s,t,n}(\tau) W_{-s,t} > 1
 \right)}_+
\\&\rightarrow 
 \pb{ \bigvee_{i = 1}^d   \pp{a^{(m)}_{it,s}}^{\alpha} \tau_i /A_i  -  \bigvee_{k = 1}^{2m}  \bigvee_{i = 1}^d  \pp{a^{(m)}_{it,k+s}}^{\alpha} \tau_i/A_i}_+ ,
\end{align}
as $n\rightarrow\infty$, where the last limit relation can be derived following similar arguments as in the proof of  \ref{item:Y_k tail}. Setting $x_s=\bigvee_{i = 1}^d   \pp{a^{(m)}_{it,s}}^{\alpha} \tau_i /A_i$,    using the identity $(x-y)_+=\max(x,y)-y$, $x,y\in \R$ and the fact that $a_{it,k+s}^{(m)}=0$ if $k+s>m$, the limit  displayed above is equal to
\[
\pb{x_s-\max_{r>s} x_r}_+=\max_{r\ge s} x_r- \max_{r\ge s+1 }x_r.
\]
So \eqref{al:joint-prob-3} follows from the telescopic sum 
$$\sum_{s=-m}^m \pb{x_s-\max_{r>s} x_r}_+= \max_{r\ge -m} x_r-\max_{r\ge m+1} x_r = \bigvee_{|r| \leq m} \bigvee_{i = 1}^d (a_{it,r}^\alpha \tau_i)/A_i.$$
\end{proof}

\begin{proof}[Proof of Theorem \ref{Thm:limit max MA}]
We shall apply Lemma \ref{Lem:approx}, for which we need to verify the conditions \eqref{eq:approx Lem cond 1} and \eqref{eq:approx Lem cond 2} there. 

\noindent
\textit{Proof of \eqref{eq:approx Lem cond 1}:}

By \eqref{eq:Y_k,i marg tail}, we have for any $k\in \Z$ and $i=1,\dots,d$ that
\begin{align}
&
\lim_{\epsilon\rightarrow 0}
\limsup_{n\rightarrow\infty }n 
\P\pp{(1-\epsilon)u_{n,i}(\tau_i) <Y^{\mathrm{M4}}_{0,i} \le (1+\epsilon)u_{n,i}(\tau_i)}
\nonumber\leq  
\lim_{\epsilon\rightarrow 0} A_i \left((1 - \epsilon)^{-\alpha} - (1 + \epsilon)^{-\alpha}\right) \tau_i
=0.
\end{align}

\noindent\textit{Proof of \eqref{eq:approx Lem cond 2}:}

Applying \eqref{eq:Y_k,i marg tail}, we have 
\begin{equation}
    \P\pp{Y_k^{\MF,(-m)} >y} \sim A_i^{(-m)} y^{-\alpha}  L(y) 
\end{equation}
as $y\rightarrow\infty$, where 
$
A_i^{(-m)}= \sum_{|r|>m}\sum_{j=1}^d a_{ij,r}^\alpha.
$
Therefore, for any $\delta>0$, we have
\begin{align} \label{eq:eq:approx Lem cond 2-proof-eq1}
&\limsup_n n\P\pp{ | Y^{\mathrm{M4}}_{0,i} - Y^{\MF,(m)}_{0,i} | > \delta u_{n,i}(\tau_i)}
 \le \limsup_{n\rightarrow\infty} n \P\pp{ Y_{0,i}^{\MF,(-m)}  > \delta u_{n,i}(\tau_i)}\\
 &\le  \lim_{n \to \infty} n\sum_{ |r|>m } \sum_{j = 1}^d  \frac{\P\left(  a_{ij,r}  W  > \delta u_{n,i}(\tau_i) \right)}{\P\pp{ W>u_{n,i}(\tau_i) }} n\P\pp{ W>u_{n,i}(\tau_i) }
=  \frac{A_i^{(-m)}\tau_i}{A_i}  \delta^{-\alpha},
\end{align}
where the last limit can be justified based on \eqref{eq:W u_n tail} and regular variation of $W$, following an argument similar as in the proof of Lemma \ref{Lem:MA tail RV}.
Hence \eqref{eq:approx Lem cond 2} follows since $A^{(-m)}_i\rightarrow 0$ as $m\rightarrow\infty$.

With conditions \eqref{eq:approx Lem cond 1} and \eqref{eq:approx Lem cond 2} proved, we shall apply Lemma \ref{Lem:approx} to obtain the desired conclusion.  
In particular,   consider the truncated process $\{Y_{k}^{\MF,(m)} \}$ in \eqref{eq:Y-m,Ym}, where each $Y_{k}^{\MF,(m)}$ is a measurable function of the underlying Gaussian  $(X_{k-m},\ldots,X_{k+m})$.
By Proposition \ref{Pro:m ext dep gaus sub multi},
the process is $2m$-extremally-dependent. Then, applying Theorem \ref{Thm:max} to the  $\{Y_{k}^{\MF,(m)} \}$   yields convergence 
\begin{equation} \label{eq:sample_extreme_max_res m}
    \lim_{n \to \infty} \P\left(  M_{n}^{(2m)} \leq u_{n}(\tau) \right) 
    =
    H_{2m}(\tau)^{\theta_{2m}(\tau)},
\end{equation}
where $M_{n}^{(2m)}$ is defined as $M_n$ in \eqref{eq:M_n component} but with $Y_k$ replaced by $Y_{k}^{\MF,(m)}$, and 
\[
H_{2m}(\tau)=  \exp\pc{ - \sum_{|r| \le m } \sum_{j =1}^d \bigvee_{i = 1}^d  \frac{a_{ij,r}^\alpha \tau_i}{A_i}
}
\]
in view of \eqref{eq:approx Lem cond 1-proof-eq3 m} and \eqref{eq:limit G(tau)}, and $\theta_{2m}(\tau)$ is  as in  \eqref{eq:theta_2m}. Now to conclude Theorem \ref{Thm:limit max MA} based on \eqref{eq:tri approx max}, it suffices to note that  $H_{2m}(\tau)\rightarrow H(\tau)$ in \eqref{eq:limit G(tau) max MA}, and $\theta_{2m}(\tau)\rightarrow \theta(\tau)$ in \eqref{eq:move-max-extremal-index}, as $m\rightarrow\infty$.
\end{proof}

\subsection{Proof of Theorem \ref{Thm:gap point proc}}\label{sec:pf Thm:gap point proc}

Let $\cal{I}$ be the collection of finite unions of sets, each of the form $\Delta\times T $, where $\Delta$ is a bounded Borel subset of $E$ with $\Lambda_{r,p}\pp{\partial \Delta}=0$, and $T$ is a bounded interval on $[0,\infty)$. 
The collection  $\cal{I}$ forms a dissecting semi-ring  on  $E\times [0,\infty)$  by \citet[Lemma 1.9]{kallenberg2017random} and hence
in view of \citet[Theorem 4.18]{kallenberg2017random}, it suffices to show for any $I\in \cal{I}$, as $n\rightarrow\infty$, 
\begin{equation}\label{eq:mean}
\esp \xi_{n} (I)  \rightarrow \esp \xi  (I)
\end{equation}
and
\begin{equation}\label{eq:zero prob}
\proba\pp{\xi_{n} (I)=0 } \rightarrow  \proba\pp{\xi  (I)=0}=\exp\pp{-\mu_r(I)},
\end{equation}
where $$\mu_r=\Lambda_{r,p}\times \Leb,$$  $\Lambda_{r,p}$ as in \eqref{eq:block conv}.
Verification of the relation \eqref{eq:mean}, and in fact, more generally, the vague convergence of the mean measures $\esp \xi_{n}$ to $\mu_r=\esp \xi $,  is standard based on the assumption \eqref{eq:block conv}, whose details we omit; see, e.g., the proof of \citet[Theorem 7.2.1]{kulik20heavy}.

We verify the relation \eqref{eq:zero prob}   for a fixed $I \in\cal{I}$ below. 
 One can construct a partition of $I$:
\begin{equation}\label{eq:I partition}
I=\bigcup_{\ell=1}^k \pp{\Delta_\ell  \times  T_\ell}=:\bigcup_{\ell=1}^k  I_\ell,
\end{equation}
where 
each $\Delta_\ell$ satisfies $\Lambda_{r,p}(\partial\Delta_\ell)=0$,  and 
$T_\ell$'s are disjoint bounded   intervals with positive lengths ordered from left to right as $\ell$ increases. Observe that $I\subset B \times [0,L] $ for some  bounded  Borel set $B\subset E$ with finite $\Lambda_{r,p}(B)$ value and some finite $L>0$. Fix $\epsilon\in (0,  \Lambda_{r,p}(B)L/2 ]$.  The partition $\{T_\ell\}$ can be refined,  if necessary,  to   ensure 
\begin{equation}\label{eq:max Leb T_ell}
\max_{1\le \ell \le k}\Leb(T_\ell)\le (2\varepsilon)/(\Lambda_{r,p}(B)^2L)\le  1/\Lambda_{r,p}(B), 
\end{equation}
so that   
\begin{equation}\label{eq:quadratic control}
\sum_{\ell=1}^k \frac{1}{2} \Lambda_{r,p}(\Delta_\ell)^2 \Leb(T_\ell)^2 \le  \frac{1}{2}\Lambda_{r,p}(B)^2 \pp{\max_{1\le \ell \le k}\Leb(T_\ell)} L \le \varepsilon
\end{equation}
and 
\begin{equation}\label{eq:mu_r I_ell <= 1}
\mu_r(I_\ell)\le  \Lambda_{r,p}(B) \Leb(T_\ell) \le 1, \quad \ell=1,\ldots,k. 
\end{equation}

Now, for $Y_{n}(i)$ as in \eqref{eq:Y_n(j)}, introduce
\begin{equation} \label{eq:def:Z_ell}
 Z_{\ell,n}(i)  =\ind\{Y_{n}(i)\in \Delta_\ell \}, \quad i=1,\ldots,n,\ \ell=1,\ldots,k.
\end{equation}
Then, the point process \eqref{eq:pointproces-Y_n(j)} can be written as
\begin{align} 
\xi_{n}(I_\ell )=\sum_{i: i(p+r)/n \in T_\ell } Z_{\ell,n}(i).
\end{align} 
By a telescopic decomposition, we get
\begin{align}
& 
\ind\{\xi_{n} (I)=0 \} -\prod_{\ell=1}^k \pp{1-  \mu_r(I_\ell)} 
\nonumber
\\&=  \prod_{\ell=1}^k \ind\{\xi_{n} (I_\ell )=0\}  -\prod_{\ell=1}^k \pp{1-  \mu_r(I_\ell)} 
\notag 
\\&= \sum_{s=1}^{k} \pp{\prod_{\ell=1}^{s-1} \ind\{\xi_{n} (I_\ell )=0\}} \Big[  \ind\{\xi_{n} (I_s )=0\}-  \pp{1-  \mu_r(I_s)}   \Big]  \pp{\prod_{\ell=s+1}^{k} \pp{1-  \mu_r(I_\ell)}}.
\label{eq:I=0 tele decomp}
\end{align}
For the $s$th term in \eqref{eq:I=0 tele decomp}, the middle factor    can be expressed as
\begin{equation} \label{eq:decomposition:EFG}
   \ind\{\xi_{n} (I_s )=0\}-  \pp{1-  \mu_r(I_s)}    =  E_{s,n} +F_{s,n}+G_{s,n},
\end{equation}
with deterministic part
\[
E_{s,n}:=  \pp{1- \esp\xi_{n}(I_s )} -\pp{1-  \mu_r(I_s)} ,
\]
and the two random variables
\[
F_{s,n}:= \esp \xi_{n}(I_s )- \xi_{n}(I_s )   ,\quad
G_{s,n}:= \ind\{\xi_{n}(I_s )=0\} -1 + \xi_{n}(I_s )\ge 0.
\]
We consider the three summands in \eqref{eq:decomposition:EFG} separately.  Note that  $\Lambda_{r,p}(\partial\Delta_s)=0$ and $\Leb( \partial T_s )=0$    imply $\mu_r(\partial I_s)=0$, $s=1,\ldots,k$. So by  \eqref{eq:mean}, we have that the first summand in \eqref{eq:decomposition:EFG} satisfies
\begin{equation}\label{eq:E_mn vanish}
\lim_{n\to \infty}   E_{s,n} = 0.
\end{equation}
For $F_{s,n}$ in \eqref{eq:decomposition:EFG}, introduce
\begin{equation}\label{eq:filteration}
\cal{F}_t=\sigma(\varepsilon_j, j\le tr+(t-1)p), \ t\in \Z,
\end{equation}
where  $\{\varepsilon_j\}$ are the i.i.d.\ Gaussian innovations of the linear process \eqref{equality_general_linear_process}. In view of \eqref{eq:Y_n(j)},
the indicator $ Z_{\ell,n}(i)$ in \eqref{eq:def:Z_ell} is a measurable function of 
$$
\pp{    Y_{(i-1)(r+p)+1} ,\ldots,     Y_{ir+(i-1)p}}.
$$ 
Therefore, due to the assumption \eqref{eq:gaus sub Y multi} and the linear representation \eqref{equality_general_linear_process}, $Z_{\ell,n}(i)$ is measurable with respect to $\cl{F}_{t}$ if $i\le t$. 
Recall  that each $I_s=\Delta_s\times T_s$ and that $T_s$'s are ordered disjoint intervals as described below \eqref{eq:I partition}. Set $t(s,n)=\max\{i\in \Z:\  i(p+r)/n\in  T_{s-1}\}$. Then  each $\xi_n(I_\ell)=\sum_{\frac{i(p+r)}{n}\in T_\ell } Z_{\ell,n}(i)$, $\ell\le s-1$, is measurable with respect to $\mathcal{F}_{t(s,n)}$. So
\begin{equation}\label{eq:F_mn vanish}
\lim_{n\to \infty} \left| \esp\left[\pp{\prod_{\ell=1}^{s-1} \ind\{\xi_{n}(I_\ell )=0\}} F_{s,n} \right] \right|\le \lim_{n\to \infty} \esp\left|\esp\left[   F_{s,n} \mid  \mathcal{F}_{t(s,n)} \right] \right|   =0,
\end{equation}
where  the last relation is due to stationarity, a triangle inequality and the relation \eqref{eq:le_statement1} in Lemma \ref{Lem:key} below.

For the third summand $G_{s,n}$ in \eqref{eq:decomposition:EFG}, note that $G_{s,n}=0$ if $\xi_n(I_s)=0$ or $1$,  and $G_{s,n}=\xi_n(I_s)-1$ otherwise. Hence,   
\[
G_{s,n}\le {\xi_n(I_s) \choose 2 }=\sum_{i<j, \, (\frac{i(p+r)}{n},\frac{j(p+r)}{n})\in T_s\times T_s } Z_{s,n}(i) Z_{s,n}(j),
\]
such that,
\begin{align}
&\limsup_{n\to \infty}\esp\left[\pp{\prod_{\ell=1}^{s-1} \ind\{\xi_{n}(I_\ell )=0\}} G_{s,n} \right]\notag \\
&\le
\limsup_{n\to \infty} \esp \left[ \sum_{i<j, \, \left(\frac{i(p+r)}{n},\frac{j(p+r)}{n}\right)\in T_s\times T_s} Z_{\ell,n}(i) Z_{\ell,n}(j)\right]\le  
\frac{1}{2} \Lambda_{r,p}(\Delta_s)^2 \Leb(T_s)^2,
\label{eq:G_mn bound}
\end{align}
where the inequality \eqref{eq:G_mn bound} follows from Lemma \ref{Lem:key} below.

Combining \eqref{eq:quadratic control}, \eqref{eq:mu_r I_ell <= 1}, \eqref{eq:I=0 tele decomp}, \eqref{eq:decomposition:EFG}, \eqref{eq:E_mn vanish}, \eqref{eq:F_mn vanish} and \eqref{eq:G_mn bound}, we conclude that
\[
\limsup_{n\to \infty} \left| \proba\pp{\xi_{n}(I)=0 } -\prod_{\ell=1}^k \pp{1-  \mu_r(I_\ell)}\right|\le \epsilon.
\]
It follows from  \eqref{eq:mean} and the elementary inequality $  1-x  \le e^{-x}\le1-x+x^2/2 $, $x\ge 0$, that
\[
\prod_{\ell=1}^k \pp{1-  \mu_r(I_\ell)}=\exp(- \mu_r(I))+o(1),
\]
where $o(1)$ is as $\varepsilon\rightarrow 0$, once noticing from \eqref{eq:max Leb T_ell} that $\max_{\ell } \mu_r(I_\ell)\le \Lambda_{r,p}(B) \max_{\ell}  \pp{\Leb\pp{T_\ell}} \le  (2\epsilon)/(\Lambda_{r,p}(B)L)  $. 
Hence, the desired relation \eqref{eq:zero prob} follows. 

\begin{Lem}\label{Lem:key}
Under the setup of Theorem \ref{Thm:gap point proc},
suppose $T$ is a bounded interval with $\inf T\ge 0$, and $\Delta$ is a bounded Borel set in $E$ with $\Lambda_{r,p}(\partial \Delta)=0$.
Introduce
\begin{equation}
 Z_n(i)  =
\ind\{Y_{n}(i)\in \Delta \}, \quad i=1,\ldots,n,
\end{equation}
where $Y_n(i)$ is as in \eqref{eq:Y_n(j)}.
Then,
\begin{align} \label{eq:le_statement1}
 \lim_{n\to \infty}   \sum_{i(p+r)/n\in T} \|\esp\left[     Z_n(i)   \mid  \mathcal{F}_{0}  \right] - \esp Z_n(i)  \|_1   =0,
\end{align}
where $\cl{F}_0$ is as in \eqref{eq:filteration}.
Furthermore, 
\begin{align} \label{eq:le_statement2}
\limsup_{n\to \infty} \esp \left[ \sum_{(i(p+r)/n,j(p+r)/n)\in T\times T, \ i<j } Z_n(i)Z_{n}(j)\right] \le  \frac{1}{2} \Lambda_{r,p}(   \Delta) ^2 \Leb(T)^2.
\end{align}
\end{Lem}

\begin{proof}

\noindent\textit{Proof of \eqref{eq:le_statement1}}.

Suppose $T\subset [0,U]$ for some $U\in (0,\infty)$.
By \eqref{eq:block conv}, for any $i\in \Z$, as $n\rightarrow\infty$, 
we have 
\begin{equation} \label{al:convergenceesp}
n\esp Z_n(i) = 
   n
\proba ( Y_{n}(i) \in \Delta)
\to 
(r+p) \Lambda_{r,p}(\Delta) .
\end{equation}
We first prove \eqref{eq:le_statement1}. 
Fix $\delta\in (0,U/(p+r)) $. We bound the left-hand side of \eqref{eq:le_statement1} via triangle   inequality  by
\begin{equation*}
 \sum_{0\le i \le \delta n }  \left\|   \esp\left[  Z_n(i)  \mid  \mathcal{F}_{0} \right]  - \esp Z_n(i)  \right\|_1 +  \sum_{\delta n<i\le  U n/(p+r) } \left\|   \esp\left[   Z_n(i)    \mid  \mathcal{F}_{0}  \right] - \esp Z_n(i)  \right\|_1 =: \mathrm{I}_n + \mathrm{II}_n.
\end{equation*}
By      \eqref{al:convergenceesp}, we have
\begin{equation}\label{eq:term I delta}
\limsup_{n\rightarrow\infty} \mathrm{I}_n\le \limsup_{n\rightarrow\infty} \sum_{0\le i\le \delta n} 2 \| Z_n(i)\|_1 \le 2\delta (r+p) \Lambda_{r,p}(\Delta).
\end{equation} 

On the other hand, by  triangle inequality and bounding the $L^1$ norms with $L^2$ norms,
\begin{align}\label{eq:II_n bound}
\mathrm{II}_n\le  \sum_{\delta n<i\le Un/(p+r)  } \| \esp\left[ Z_n(i) \mid   \cal{F}_0 \right] -\esp Z_n(i)  \|_2.
\end{align}
Recall that $\cal{F}_0=\sigma(\varepsilon_i: \  i\le -p )$, and $ Z_n(i)$ is a measurable function of
\begin{equation}\label{eq:bf X_i}
X(i)=X_{r,p,m}(i):=\pp{    X_{(i-1)(r+p)-m+1} ,\ldots,     X_{ir+(i-1)p}}.
\end{equation}
We aim to apply Proposition \ref{Pro:proj mehler}  with $H_1$ and $H_2$  as the closed subspaces spanned by $X(i)$ and  $(\varepsilon_j: \  j\le -p )$, respectively, and one can take the total space $H$ as spanned by $\{ \varepsilon_i \}_{i\in \Z}$.  Then,   
\begin{align}\label{eq:apply Q bound}
    \| \esp\left[ Z_n(i)\mid   \cal{F}_0  \right] -\esp Z_n(i)  \|_2^2\le \Var\pp{M_{\rho_i} ( Z_n(i)-\esp Z_n(i)) },
\end{align}
where $M_{\rho_i}$ is the Mehler transform with respect to $H$ and $\rho_i$ is the maximal correlation between $H_1$ and $H_2$; see \eqref{eq:M_a}. In view of Proposition \ref{Pro:rho equal}, we can identify $\rho_i$ with the norm of the operator $P_{H_1H_2}$ as 
\begin{equation}\label{eq:rho_i opt}
\rho_i= \sup \| \esp[ a_1' X_{(i-1)(r+p)-m+1}+\ldots    +a_{r+m}'  X_{ir+(i-1)p} \mid \cal{F}_0]\|_2,
\end{equation}
where the supremum above is taken over all real vectors $a_1,\ldots,a_{r+m}$ such that 
\begin{equation}\label{eq:constraint} 
\|a_1' X_{(i-1)(r+p)-m+1}+\ldots    + a'_{r+m}  X_{ir+(i-1)p}\|_2^2\le 1.
\end{equation}
Note that for $t\ge -p$, we have
$\esp[  X_{t}    \mid \cal{F}_0]= \sum_{s\le -p} \Psi_{t-s}\varepsilon_{s}$.
Therefore,    for $i>\delta n$ and $n$ large enough, we have
\begin{align}
   &   \left\| \sum_{j=1}^{r+m} a'_j \esp\left[ X_{(i-1)(r+p)-m+j}  \mid \cal{F}_0\right]\right\|_2^2 
   =   \left\|  \sum_{s\le -p}   \sum_{j=1}^{r+m}  a'_j \Psi_{(i-1)(r+p)-m+j-s}\varepsilon_s \right\|_2^2 \notag \\
   &=  \sum_{s\le -p}  
   \left\|\sum_{j=1}^{r+m}  a'_j \Psi_{(i-1)(r+p)-m+j-s}\right\|^2\le (r+m) \sum_{s\le -p}  
\sum_{j=1}^{r+m}     \left\| a'_j \Psi_{(i-1)(r+p)-m+j-s}\right\|^2
   \nonumber
   \\&\le   
   (r+m) \pp{\sum_{j=1}^{r+m} \| a_j\|^2}  \pp{ \sum_{i> c   n}  \|\Psi_i'\|_{\mathrm{op}}^2 }, \label{eq:upp bound}
\end{align} 
where $ \|\cdot\|_{\mathrm{op}}$ denotes the operator (or spectral) norm of a matrix, and  $c>0$ can be taken as $ \delta(r+p)/2$.
On the other hand, 
\begin{align}\label{eq:lower bound}
    \|a_1'X_{(i-1)(r+p)-m+1}+\ldots  +  a_{r+m}'  X_{ir+(i-1)p}\|_2^2
\ge  
     \pp{\|a_1\|^2+\ldots+\|a_{r+m}\|^2}\lambda_{\min} (r+m),
\end{align}
where $\lambda_{\min} (r+m)$ is the minimum eigenvalue of the covariance matrix $\pp{\Gamma(i-j)}_{1\le i,j\le r+m}$. Note that $\lambda_{\min} (r+m)>0$ by Lemma \ref{le:multiBrochwellDavis} since   $\Gamma_{jj}(h)\rightarrow 0$ as $h\rightarrow \infty$ for the linear moving average \eqref{equality_general_linear_process}.
Combining \eqref{eq:rho_i opt}, \eqref{eq:constraint}, \eqref{eq:upp bound}   and \eqref{eq:lower bound},  we conclude for    $n$ sufficiently large that
\begin{equation}\label{eq:bar rho_n}
  \sup_{i>\delta n} \rho_i^2 \le \left[   \frac{r+m}{\lambda_{\min}(r+m)} \pp{  \sum_{i>c   n}   \|\Psi_i'\|^2_{\mathrm{op}} }\right]  \wedge 1=:\widebar{\rho}_n^2.
\end{equation}
Set $\ell_n=\sqrt{1/\log(n)}$, $n\ge 3$.  So by Lemma \ref{le:behavior_acf}, we have
\begin{equation}\label{eq:bar rho ell}
\widebar{\rho}_n^2/\ell_n^2\rightarrow 0,
\end{equation}
as $n\rightarrow\infty$.
Hence in view of \eqref{eq:apply Q bound}, \eqref{eq:bar rho_n} and applying \eqref{eq:Mehler mono} twice,
we have    for $i>\delta n$ and $n$  large enough that 
\begin{equation}\label{eq:Q Z_n(i) bound}
   \| \esp\left[ Z_n(i)\mid   \cal{F}_0  \right] -\esp Z_n(i)  \|_2^2 \le  \var\pp{ M_{\widebar{\rho}_n}( Z_n(i))} \le  \frac{\widebar{\rho}_n^2}{\ell_n^2} \var\pp{ M_{\ell_n}( Z_n(i))}.
   \end{equation}
On the other hand, by \eqref{eq:M_r hypercontr}, we have
\begin{align}\label{eq:var M bound}
 \var\pp{ M_{\ell_n}( Z_n(i))}  \le  \|    M_{\ell_n}( Z_n(i))  \|_2^2\le     \| Z_n(i)\|_{1+\ell_n^2}^2 =O(n^{-2/(1+\ell_n^2)}),
\end{align}
as $n\rightarrow\infty$, where the last step follows from \eqref{al:convergenceesp}  and the   fact that 
\[
n^{2-2/(1+\ell_n^2)}=\exp\pp{[2\ell_n^2/(1+\ell_n^2)]\log(n)}\rightarrow e^{2},
\] 
as $n\rightarrow\infty$.  Hence  putting together \eqref{eq:II_n bound}, \eqref{eq:bar rho ell}, \eqref{eq:Q Z_n(i) bound} and \eqref{eq:var M bound}, we conclude
\[
\mathrm{II}_n\rightarrow 0,
\]
as $n\rightarrow\infty$.  Combining this and \eqref{eq:term I delta} where $\delta>0$ can be arbitrarily small, we establish \eqref{eq:le_statement1}.

\medskip
 \noindent\textit{Proof of \eqref{eq:le_statement2}}.

  By \eqref{eq:hyper CS} and stationarity, we have for any $ i<j$,
\[
\esp Z_n(i) Z_n(j) \le 
(\P\pp{Y_{n}(1)\in \Delta } )^{2/\pp{1+\gamma_X(j-i)}},
\]
where 
\[
\gamma_X(h)=\rho(X(h), X(0) ), \quad h\ge 0,
\]
with $X(h)$ as in \eqref{eq:bf X_i}, and  $\rho(\cdot,\cdot)$ stands for the canonical correlation; see Proposition \ref{Pro:rho equal}. In view of \eqref{al:convergenceesp}, it remains to show
\begin{equation}\label{eq:conv to half}
\sum_{(i(r+p)/n,j(r+p)/n)\in T\times T, \ i<j }
\pp{ (r+p) \Lambda_{r,p}(\Delta) n^{-1} }^{2/ (1+ \gamma_X(j-i) )}\rightarrow \frac{1}{2} \Lambda_{r,p}(   \Delta) ^2 \Leb\pp{T}^2, 
\end{equation}
as $n\rightarrow\infty$. Recall the covariance function $\Gamma$ for $\{ X_i \}$ satisfies $\Gamma_{jj}(h)\rightarrow 0$  as $h\rightarrow\infty$.
 This implies 
 \begin{equation}\label{eq:gamma_Y<1}
 \gamma_X(h)<1, \quad \text{ for all } h\ge 1;
 \end{equation}
 otherwise,  a linear combination of the components of $X(h)$ is equal to a linear combination of the components of $X(0)$ almost surely, a contradiction in view of Proposition \ref{Pro:m ext dep gaus sub multi}.
  On the other hand, by \citet[Lemma 3.4]{bai2016unified},  the assumption \eqref{eq:Gammamultiasymp} and Lemma \ref{le:multiBrochwellDavis}, we have
\begin{align}\label{eq:gamma_Y log rate}
    \gamma_X(h)
    = o\pp{ 1/\log(h)}, 
\end{align}
as $h\rightarrow\infty$. Now, set $C=(r+p)\Lambda_{r,p}(\Delta)$, and let $H(n)$ denote the difference between the largest and smallest integers in the interval $Tn/(r+p)$. Note that 
\begin{equation}\label{eq:H asymp}
\lim_{n \to \infty} \frac{H(n)}{n}= \frac{\Leb(T)}{(r+p)}, \quad \text{and } \quad 
\frac{H(n)}{n}\le \frac{\Leb(T)}{(r+p)}+2=: \widebar{H} 
\end{equation}
for all $n\in \Z_+$. Then the left-hand side of \eqref{eq:conv to half} can be expressed as
\begin{align}
&
\frac{C^2}{n} \sum_{  k=1 }^{H(n)} \left( \frac{H(n)}{n}- \frac{k}{n} \right) \pp{\frac{n}{C}}^{ 2\gamma_X(k) / (1+ \gamma_X(k) ) }=
C^2\int_{0}^{\infty} f_{n}(x) dx,
\label{al:second_claim_al2.2}
\end{align}
where
\begin{align}\label{eq:f_n}
f_{n}(x) = \left( \frac{H(n)}{n} - \frac{\lceil xn \rceil}{n} \right)_+ \exp\left\{\frac{2\log ( n /C )   \gamma_{X}( \lceil xn \rceil ) }  {1 +  \gamma_{X}( \lceil xn \rceil)  } \right\}.
\end{align} 
 To apply the dominated convergence theorem, we need to separate a singularity near $x=0$.   Due to \eqref{eq:gamma_Y<1} and \eqref{eq:gamma_Y log rate},  we have $\beta:=\sup_{k\ge 1}2\gamma_X(k)/(1+\gamma_X(k))<1$. 
Choose $\beta_0\in (\beta,1)$. Then, using monotonicity, we have
\begin{equation*} 
\int_0^{n^{-\beta_0}} f_n(x) dx\le  \widebar{H}   n^{1-\beta_0}  \pp{\frac{n}{C}}^{\beta-1}\rightarrow 0 ,
\end{equation*}
as $n\rightarrow\infty$.
We then apply the dominated convergence theorem to justify the convergence 
\begin{equation}\label{eq:DCT term}
\int_{0}^\infty \ind\{x >n^{-\beta_0}\} f_n(x) dx\rightarrow \frac{1}{2}  \pp{\frac{\Leb(T)}{r+p}}^2,
\end{equation}
as $n\rightarrow\infty$, and hence ultimately \eqref{eq:conv to half} when combining the relations above.  In fact, in view of \eqref{eq:gamma_Y log rate},   for all $x>0$, the expression of the exponent within $\exp   \{\cdots\}$ in \eqref{eq:f_n}  converges to $0$. Therefore for any $x>0$, taking into account also \eqref{eq:H asymp} we have $f_{n}(x)\rightarrow  (\Leb(T)/(r+p) -x)_+$,  as $n\rightarrow\infty$, whose integral is exactly the right-hand side of \eqref{eq:DCT term}. 
Next,    the choice $\beta_0<1$  and  \eqref{eq:gamma_Y log rate} imply that  
$$
C_0:=\sup_{n\ge 1,\, x>n^{-\beta_0}}  \log(n/C) \gamma_X(\lceil xn \rceil )<\infty.
$$
So in view of \eqref{eq:f_n}, a dominating bound for the integrand in \eqref{eq:DCT term}  can be found as 
\begin{align*}
\ind\{x>n^{-\beta_0} \}|f_{n}(x)| \leq \widebar{H}  \exp\left\{2 C_0 \right\} \ind\{ x\in [0, \widebar{H}] \}. 
\end{align*}
\end{proof}

\subsection{Proof of Proposition \ref{Pro:D'(u)}}\label{sec:pf D'}

\begin{proof}
For the condition $D'(u_n)$, note that by Proposition \ref{Pro:hyper holder},
\begin{align*}
\P\pp{Y_0>u_n, Y_j>u_n}
=
\esp{ \pp{ \mathbf{1}\{Y_0>u_n\} \mathbf{1}\{ Y_j>u_n\} }}
\le 
\P\pp{ Y_0>u_n }^{2/\pp{1+\gamma_X(j)}}.
\end{align*}
with   $\gamma_X(h)=\esp X(h) X(0) $.
Using $\lim_{n \to \infty} n\P\pp{Y_0>u_n}=\tau$, we get for some constant $C>0$ that
\begin{align*}
    n\sum_{1 \le j\le n/k} \proba\pp{Y_0>u_n, Y_j>u_n }&\le C \sum_{1\le j\le n/k} n^{-1}     n^{2 \gamma_X(j)/(1+\gamma_X(j))}\\
    &\le C \int_0^{1/k} f_n^*(x)dx,
\end{align*}
where
\[
f^*_n(x)=\exp\{ \log(n)  2\gamma_X( \lceil xn \rceil )  / \left(1 +  \gamma_X( \lceil xn \rceil)  \right)   \},
\]
with $\lceil z \rceil$ denoting  the smallest integer no less than $z\in \R$.
Similarly as how $f_n$ in \eqref{eq:f_n} is handled in the proof of Lemma \ref{Lem:key}, with $\gamma_X(n)=o(\log(n)^{-1})$, we can show by the dominated convergence theorem combined with a truncation argument that
\[
\int_0^{1/k} f_n^*(x)dx \rightarrow \int_0^{1/k} 1 \ dx=1/k, 
\]
as $n\rightarrow\infty$. Hence \eqref{eq:D_n'} holds.
\end{proof}
In fact, Proposition \ref{Pro:D'(u)} can be further extended to the case where $Y_k$ is given as one of the components in \eqref{eq:gaus sub Y multi}, as long as all the cross-covariances of $\{X_k\}$ decay to zero faster than an inverse logarithm as the time lag tends to infinity. For the proof, we can simply replace $\gamma_X(h)$ in the proof above by the canonical correlation in \eqref{eq:gamma_Y log rate}.

\small
\bibliographystyle{plainnat}
\bibliography{references}

\end{document}